\documentclass{amsart}
\usepackage{amssymb}
\usepackage{amsmath}
\usepackage{latexsym}
\usepackage{graphicx}
\usepackage{amscd}
\usepackage{mathrsfs}
\usepackage[english]{babel}
\usepackage{hyperref}
\usepackage{verbatim}
\usepackage{tikz}
\usetikzlibrary{math}

\input xy
\xyoption{all}

\newtheorem{theorem}[equation]{Theorem}
\newtheorem{corollary}[equation]{Corollary}
\newtheorem{prop}[equation]{Proposition}
\newtheorem{lemma}[equation]{Lemma}
\newtheorem{cor}[equation]{Corollary}

\theoremstyle{definition}

\newtheorem{exam}[equation]{Example}
\newtheorem{example}[equation]{Example}

\newtheorem{remark}[equation]{Remark}

\newtheorem*{theorem*}{Main Theorem}

\newcommand{\llbr}{[\negthinspace[}
\newcommand{\rrbr}{]\negthinspace]}

\newcommand{\llpar}{(\negthinspace(}
\newcommand{\rrpar}{)\negthinspace)}

\newcommand{\A}{\ensuremath{\mathbb{A}}}
\newcommand{\LL}{\ensuremath{\mathbb{L}}}
\newcommand{\PP}{\ensuremath{\mathbb{P}}}
\newcommand{\N}{\ensuremath{\mathbb{N}}}
\newcommand{\Z}{\ensuremath{\mathbb{Z}}}
\newcommand{\Q}{\ensuremath{\mathbb{Q}}}

\newcommand{\R}{\ensuremath{\mathbb{R}}}
\newcommand{\C}{\ensuremath{\mathbb{C}}}

\newcommand{\Pro}{\ensuremath{\mathbb{P}}}

\newcommand{\cU}{\ensuremath{\mathscr{U}}}

\newcommand{\cX}{\ensuremath{\mathscr{X}}}
\newcommand{\cY}{\ensuremath{\mathscr{Y}}}
\newcommand{\cZ}{\ensuremath{\mathscr{Z}}}

\newcommand{\Spec}{\ensuremath{\mathrm{Spec}\,}}

\newcommand{\red}{\mathrm{red}}

\newcommand{\Var}{\mathrm{Var}}
\newcommand{\Vardim}{\mathrm{Var}^{\mathrm{dim}}}

\newcommand{\Vol}{\mathrm{Vol}}
\newcommand{\Volsb}{\mathrm{Vol}_{\sbir}}
\newcommand{\Volbir}{\mathrm{Vol}_{\bir}}

\newcommand{\Gro}{\mathbf{K}}
\newcommand{\Grodim}{\mathbf{K}}
\newcommand{\Bir}{\mathrm{Bir}}
\newcommand{\bir}{\mathrm{bir}}
\newcommand{\SB}{\mathrm{SB}}
\newcommand{\SBdim}[1]{\mathrm{SB^{dim}_{#1}}}
\newcommand{\sbir}{\mathrm{sb}}
\newcommand{\Bitt}{\mathrm{B}}
\newcommand{\codim}{\mathrm{codim}}

\numberwithin{equation}{subsection}

\author{Johannes Nicaise}
\address{Imperial College,
Department of Mathematics, South Kensington Campus,
London SW72AZ, UK, and KU Leuven, Department of Mathematics, Celestijnenlaan 200B, 3001 Heverlee, Belgium.} \email{j.nicaise@imperial.ac.uk}

\author{John Christian Ottem}
\address{Department of Mathematics, University of Oslo, Box 1053, Blindern, 0316 Oslo, Norway}
\email{johnco@math.uio.no}

\begin{document}
\title[A refinement of the motivic volume]{A refinement of the motivic volume, and specialization of birational types}

\begin{abstract}
We construct an upgrade of the motivic volume by keeping track of dimensions in the Grothendieck ring of varieties.  
 This produces a uniform refinement of the motivic volume and its birational version introduced by Kontsevich and Tschinkel to prove the specialization of birational types. We also provide several explicit examples of obstructions to stable rationality arising from this technique.
\end{abstract}

\maketitle

\section{Introduction}\label{sec:intro}
 In \cite{NiSh}, Shinder and the first-named author used a refinement of Denef and Loeser's motivic nearby fiber to construct a motivic obstruction to the stable rationality of very general fibers of a degenerating family of smooth and proper complex varieties. This result implies in particular that stable rationality specializes in smooth and proper families. A variant of this method was developed by Kontsevich and Tschinkel in \cite{KT} to upgrade the results to rationality instead of stable rationality.
 
 The principal aim of the present article is to develop a unified framework for the invariants in \cite{NiSh} and \cite{KT}. The approach in \cite{NiSh} relies on the {\em motivic volume}, a motivic specialization morphism for Grothendieck rings of varieties which constructs a limit object for the classes of the fibers of degenerating families.
 It can be viewed as a version of the nearby cycles functor at the level of Grothendieck rings. It is closely related to the {\em motivic nearby fiber} of Denef and Loeser \cite{DL}, but with the crucial difference that its construction does not require the inversion of the class of the affine line in the Grothendieck ring. This is essential for the applications to rationality problems. The existence of the motivic volume can be deduced from the work of Hrushovski and Kazhdan \cite{HK}, but in \cite{NiSh} a direct proof was given based on the weak factorization theorem and logarithmic geometry.
 
 It was proved by Larsen and Lunts \cite{larsen-lunts} that the stable birational type of a smooth and proper variety over a field of characteristic zero can be read off from its class in the Grothendieck ring of varieties. However, it is an open problem whether this class also determines the birational type, which is why the results in \cite{NiSh} only give information about stable birational types. This issue was circumvented in \cite{KT} by working directly in the free abelian group on birational types instead of the Grothendieck ring of varieties, and constructing an analogous specialization morphism there.

 Here we introduce a refined version of the Grothendieck ring of varieties, graded by dimension, which detects birational types for trivial reasons. A variant of this graded ring already appeared in \cite{HK} in the form of the ring $!K(\mathrm{RES}[*])$. We then show that the arguments in \cite{NiSh} can be carried over {\em verbatim} to construct the motivic volume at this refined level. The result is an invariant that specializes simultaneously to the motivic volume from \cite{NiSh} and to its birational version in \cite{KT}. All the results for stable birational types in \cite{NiSh} can then immediately be upgraded to the level of birational types, yielding a unified treatment of the specialization results in \cite{NiSh} and \cite{KT}.

 A second new feature is that we present the construction of the motivic volume in a slightly more general setting (this generalization is explained in detail at the beginning of the proof of Theorem \ref{theo:MV}). We also replace as much as possible the language of logarithmic geometry by the more widely familiar language of toroidal embeddings. This makes the motivic volume more accessible and user-friendly. The proofs still  rely heavily on those in \cite{NiSh} and there logarithmic geometry remains by far the most efficient and transparent framework. Still, if the reader is willing to accept the most technical aspects of the construction as black boxes, then they can understand and use the motivic volume without any reference to logarithmic geometry.
 
 Denef and Loeser's motivic nearby fiber and the motivic volume  have found many profound applications in singularity theory and motivic Donaldson--Thomas theory; see for instance \cite{GLM, thuong, NP}. It would be interesting to investigate what additional information can be extracted from our refined version, beyond the applications to birational types discussed here.

\smallskip The paper is organized in the following way. Section \ref{sec:gro} is devoted to the Grothendieck ring of varieties. We first recall the definition of the classical Grothendieck ring of varieties, and we give an overview of its main properties, with an emphasis on the theorems of Bittner (Theorem \ref{thm:Bitt}) and Larsen \& Lunts (Theorem \ref{thm:LL}) which are the most important structural results in characteristic zero. We explain why the Grothendieck ring detects stable birational types in characteristic zero and why it is unclear whether it also detects birational types. This part is not needed for the remainder of the paper, but it places the results in their proper context. We then move on to  our refinement of the Grothendieck ring of varieties graded by dimension, and we show that it detects birational types in arbitrary characteristic (Proposition \ref{prop:grobir}). We prove analogs of the the theorems of Bittner (Theorem \ref{thm:dimBitt}) and Larsen \& Lunts (Theorem \ref{thm:dimLL}). The analog of Bittner's theorem is an essential ingredient in the construction of the motivic volume; the analog of the theorem of Larsen \& Lunts is not used in the sequel and is only included to further clarify the structure of the refined Grothendieck ring.

 In Section \ref{sec:motvol}, we present the construction of the motivic volume at the level of the refined Grothendieck ring, which is the key invariant for the applications to rationality problems. We define the motivic volume in terms of strictly toroidal models, making the construction as explicit as possible (Theorem \ref{theo:MV}). The proof of its existence is essentially the same as for the unrefined version in \cite{NiSh}; we explain the general strategy and highlight the new elements of the proof. The main applications to rationality problems are developed in Section \ref{ss:apprat}. We deduce the specialization of birational types in smooth and proper families (Theorem \ref{theo:spec}; this result was originally proved in \cite{KT}) and we explain how every strictly toroidal model gives rise to an obstruction to the (stable) rationality of the geometric generic fiber (Theorem \ref{theo:obstruct}). Some elementary  examples of such obstructions are given in Section \ref{ss:exam}; further applications can be found in \cite{NiSh} and \cite{NO}.

  Finally, in Section \ref{sec:monodromy}, we construct the monodromy action on the refined motivic volume (Theorem \ref{theo:MVmon}). This monodromy action is an important feature of the motivic volume and Denef and Loeser's motivic nearby fiber, as it captures the monodromy action on the nearby cycles complex at the motivic level. 

\subsection*{Terminology}
Let $F$ be a field, and let $X$ and $Y$ be reduced $F$-schemes of finite type.
 Then $X$ and $Y$ are called {\em birational}  if there exist a dense open subscheme $U$ of $X$ and a dense open subscheme $V$ of $Y$ such that the $F$-schemes $U$ and $V$ are isomorphic. This defines an equivalence relation on the set of isomorphism classes of reduced $F$-schemes of finite type; the equivalence class of $X$ is called its birational type. We say that $X$ is rational  if it is birational to the projective space $\mathbb{P}^d_F$ for some $d\geq 0$; this implies in particular that $X$ is integral.
 
 We say that $X$ and $Y$ are  {\em stably birational} if $X\times_{F}\mathbb{P}^m_F$ is birational to $Y\times_F \mathbb{P}^n_F$ for some $m,n\geq 0$. This again defines an equivalence relation on the set of isomorphism classes of reduced $F$-schemes of finite type; the equivalence class of $X$ is called its stable birational type. We say that $X$ is stably rational  if it is stably birational to $\Spec F$; equivalently, if $X\times_F \mathbb{P}^m_F$ is birational to $\mathbb{P}^n_F$ for some $m,n\geq 0$.
 
 It is obvious that birational schemes are also stably birational, so that rationality implies stable rationality. The converse is not true: the first counterexample was constructed in \cite{BCTSSD}.

\subsection*{Acknowledgements}
 It is a pleasure to thank the organizers of the conference {\em Rationality of Algebraic Varieties} on the beautiful Schiermonnikoog Island in April 2019 for putting together the event and inviting us to submit a contribution to the proceedings. Our investigation of a common refinement of the Grothendieck ring of varieties and the ring of birational types was motivated by questions from Gerard van der Geer and Lenny Taelman, and the very first sketch of this paper was made on the trip back from Schiermonnikoog. The first-named author would also like to thank Evgeny Shinder for the rewarding collaboration that has led to the article \cite{NiSh}, which is the basis for the present paper. We are grateful to Olivier Benoist, Jean-Louis Colliot-Th\'el\`ene, Alexander Kuznetsov, Daniel Litt, Sam Payne, Alex Perry and Stefan Schreieder for stimulating discussions during the preparation of this paper and \cite{NO}. Finally, our thanks go out to the referee for carefully reading the manuscript and making several valuable suggestions.
 
Johannes Nicaise is supported by EPSRC grant EP/S025839/1, grants G079218N of the Fund for Scientific Research--Flanders, and long term structural funding (Methusalem
grant) of the Flemish Government. John Christian Ottem is supported by the Research Council of Norway project no. 250104.

\section{The Grothendieck ring of varieties graded by dimension}\label{sec:gro}
\subsection{Reminders on the Grothendieck ring of varieties}
In this section we give a quick overview of the classical Grothendieck ring of varieties. We refer to \cite{motbook} for further background on the Grothendieck ring and its role in the theory of motivic integration.

Let $F$ be a field. The Grothendieck group $\Gro(\Var_F)$ of $F$-varieties is the abelian group with the following presentation.
 \begin{itemize}
 \item {\em Generators:} isomorphism classes $[X]$ of $F$-schemes $X$ of finite type;
 \item {\em Relations:} whenever $X$ is an $F$-scheme of finite type, and $Y$ is a closed subscheme of $X$, then $[X]=[Y]+[X\setminus Y].$
 \end{itemize}
These relations are often called {\em scissor relations} because they allow to cut a scheme into a partition by subschemes.
  The group $\Gro(\Var_F)$ has a unique ring structure such that $[X]\cdot [X']=[X\times_F X']$ for all $F$-schemes $X$ and $X'$ of finite type. The neutral element for the multiplication is $1=[\Spec F]$, the class of the point. We denote by $\LL=[\A^1_F]$ the class of the affine line in $\Gro(\Var_F)$. 

\begin{example}
Let $n$ be a positive integer. 
Partitioning $\PP^n_F$ into the hyperplane at infinity and its complement, we find 
$$[\PP^n_F]=[\PP^{n-1}_F]+[\A^n_F]=[\PP^{n-1}_F]+\LL^n.$$
Now it follows by induction on $n$ that 
$$[\PP^n_F]=1+\LL+\ldots+\LL^n$$ in $\Gro(\Var_F)$.
\end{example}

\begin{remark}
The Grothendieck ring $\Gro(\Var_F)$ is insensitive to non-reduced structures: if $X$ is an $F$-scheme of finite type and we denote by $X_{\red}$ its maximal reduced closed subscheme, then the complement of $X_{\red}$ in $X$ is empty, so that $[X]=[X_{\red}]$. Thus we would have obtained the same Grothendieck ring by taking as generators the isomorphism classes of reduced $F$-schemes of finite type (taking the fibered product in this category to define the ring multiplication). In fact, since every $F$-scheme of finite type can be partitioned into regular quasi-projective $F$-schemes, we could even have taken the isomorphism classes of such schemes as generators (but then some care is required in the definition of the product if $F$ is not perfect).
\end{remark}

The structure of the ring $\Gro(\Var_F)$ is still poorly understood. The main challenge is to characterize geometrically when two $F$-schemes $X$ and $X'$ of finite type define the same class in the Grothendieck ring of varieties. An obvious sufficient condition is that $X$ and $X'$ be piecewise isomorphic, that is, can be partitioned into subschemes that are pairwise isomorphic; then the scissor relations immediately imply that $[X]=[X']$. 

\begin{example}
Let $C\subset \A^2_F$ be the affine plane cusp over $F$, defined by the equation $y^2-x^3=0$.
 Then $C$ is piecewise isomorphic to the affine line $\A^1_F$, because $C\setminus \{(0,0)\}$ is isomorphic to $\A^1_F\setminus \{0\}$. It follows that $[C]=\LL$ in $\Gro(\Var_F)$.
\end{example}

However, this condition is not necessary: L.~Borisov has recently given an example of two $\C$-schemes $X$ and $X'$ of finite type such that $[X]=[X']$ but $X$ and $X'$ are not birational, and therefore certainly not piecewise isomorphic \cite{borisov}. This example is due to issues of cancellation: $X$ and $X'$ can be embedded into a common $\C$-scheme $W$ of finite type such that $W\setminus X$  and $W\setminus X'$ can be partitioned into pairwise isomorphic subschemes $W_1,\ldots,W_r$ and $W'_1,\ldots,W'_r$, respectively. 
It follows that $$[X]=[W]-\sum_{i=1}^r[W_i]=[W]-\sum_{i=1}^r[W'_i]=[X'],$$ even though $X$ and $X'$ are not piecewise isomorphic. The schemes $X$ and $X'$ in Borisov's example are smooth, but not proper. We do not know any such example where $X$ and $X'$ are smooth and proper, and it is still an open question whether the class in the Grothendieck ring detects the birational type of a smooth and proper $F$-scheme $Y$. We will see in Corollary \ref{cor:LL} that, when $F$ has characteristic zero, the class $[Y]$ determines the {\em stable} birational type of $Y$.  
 
 The structure of $\Gro(\Var_F)$ is particulary obscure when $F$ has positive characteristic. For instance, if $F$ is an imperfect field and $F'$ is a purely inseparable finite extension of $F$, it is not known whether $[\Spec F']$ is different from $[\Spec F]$. The situation is better in characteristic zero, thanks to resolution of singularities and the Weak Factorization Theorem.
 Let us recall two of the most powerful results in this setting: the theorems of Bittner and Larsen \& Lunts. We start with an easy consequence of Hironaka's resolution of singularities.
 
\begin{lemma}\label{lemm:Hir}
Let $F$ be a field of characteristic zero. Then the group $\Gro(\Var_F)$ is generated by the classes of smooth and proper $F$-schemes.
\end{lemma} 
\begin{proof}
Let $X$ be a reduced $F$-scheme of finite type. Since $F$ has characteristic zero, we know that $X$ is birational to a smooth and proper $F$-scheme $X'$; then there exist strict closed subschemes $Y$ and $Y'$ of $X$ and $X'$, respectively, such that $X\setminus Y$ is isomorphic to $X'\setminus Y'$.
 It follows from the scissor relations that $[X]=[X']-[Y']+[Y]$, and by induction on the dimension, we may assume that $[Y]$ and $[Y']$ can be written as 
  $\Z$-linear combinations of classes of smooth and proper $F$-schemes. 

 We can also give a slightly more involved argument that is often useful to find a more explicit expression of $[X]$ in terms of classes of smooth and proper $F$-schemes.  We can partition $X$ into separated smooth subschemes $X_1,\ldots,X_r$. Then the scissor relations imply that 
$[X]=[X_1]+\ldots+[X_r]$ in $\Gro(\Var_F)$. Thus it suffices to show that we can write $[X]$ as a linear combination of classes of smooth and proper $F$-schemes when $X$ is separated and smooth. By Hironaka's embedded resolution of singularities, we can find a smooth compactification $\overline{X}$ of $X$
 such that the boundary $\overline{X}\setminus X$ is a divisor with strict normal crossings. 
 Let $E_i,\,i\in I$ be the prime components of this divisor. For every subset $J$ of $I$, we set $E_J=\cap_{j\in J}E_j$ (in particular, $E_{\emptyset}=\overline{X}$). We endow the closed subsets $E_J$ of $\overline{X}$ 
with their induced reduced structures; then the schemes $E_J$ are smooth and proper over $F$.  Using the scissor relations and an inclusion-exclusion argument, one easily checks that 
 \begin{equation}\label{eq:compact}
 [X]=\sum_{J\subset I}(-1)^{|J|}[E_J]
 \end{equation}
 in $\Gro(\Var_F)$.
\end{proof}
 
\begin{theorem}[Bittner 2004]\label{thm:Bitt}
Let $F$ be a field of characteristic zero. We define an abelian group $\Gro^{\Bitt}(\Var_F)$ by means of the following presentation.
\begin{itemize}
\item {\em Generators:} isomorphism classes $[X]^{\Bitt}$ of connected smooth and proper $F$-schemes $X$;
\item {\em Relations:} $[\emptyset]^{\Bitt}=0$, and, whenever $X$ is a connected smooth and proper $F$-scheme and $Y$ is a connected smooth  closed subscheme of $X$,  
\begin{equation}\label{eq:blup}
[\mathrm{Bl}_YX]^{\Bitt}-[E]^{\Bitt}=[X]^{\Bitt}-[Y]^{\Bitt}
\end{equation}
where $\mathrm{Bl}_YX$ denotes the blow-up of $X$ along $Y$, and $E$ is the exceptional divisor.
\end{itemize}
For every smooth and proper $F$-scheme $X$, we set 
$$[X]^{\Bitt}=[X_1]^{\Bitt}+\ldots+[X_r]^{\Bitt}$$ where $X_1,\ldots,X_r$ are the connected components of $X$. We endow $\Gro^{\Bitt}(\Var_F)$ with the unique ring structure such that $[X]^{\Bitt}\cdot [X']^{\Bitt}=[X\times_F X']^{\Bitt}$ for all smooth and proper $F$-schemes $X$ and $X'$.

Then there exists a unique group morphism 
$$\Gro^{\Bitt}(\Var_F)\to \Gro(\Var_F)$$
that maps $[X]^{\Bitt}$ to $[X]$, for every smooth and proper $F$-scheme $X$. 
 This morphism is an isomorphism of rings.
\end{theorem}
\begin{proof}
This is the main part of Theorem 3.1 in \cite{bittner}. Let us briefly sketch the steps of the proof. The uniqueness and existence of the morphism are straightforward: the blow-up relations \eqref{eq:blup} are satisfied in 
$\Gro(\Var_F)$ because $\mathrm{Bl}_YX\setminus E$ is isomorphic to $X\setminus Y$. It is also clear that this is a morphism of rings. The surjectivity of the morphism follows from the fact that the group $\Gro(\Var_F)$ is generated by the classes of smooth and proper $F$-schemes, by Lemma \ref{lemm:Hir}.   The principal difficulty is showing that $\Gro^{\Bitt}(\Var_F)\to \Gro(\Var_F)$ is injective. This is achieved by constructing an inverse of this morphism.

 Since we can partition every $F$-scheme of finite type into separated smooth $F$-schemes $X$ of finite type, the essential step is to define the inverse on such schemes $X$ (of course, one needs to check that the final construction is independent of the choice of the partition).
 Let $\overline{X}$ be a strict normal crossings compactification of $X$ as in the proof of Lemma \ref{lemm:Hir}. 
 Then a straightforward calculation shows that the element
\begin{equation}\label{eq:Bittproof}
 \sum_{J\subset I}(-1)^{|J|}[E_J]^{\Bitt}
 \end{equation}
in $\Gro^{\Bitt}(\Var_F)$ is invariant under blow-ups of smooth centers in the boundary of $\overline{X}$ that have strict normal crossings with all the strata $E_J$. The Weak Factorization Theorem \cite{WF} implies that any two strict normal crossings compactifications of $X$ can be connected by a chain of such blow-ups and blow-downs.
 It follows that the element \eqref{eq:Bittproof} only depends on $X$, and not on the chosen strict normal crossings compactification $\overline{X}$. Now it is not hard to show that there exists a unique group morphism 
$\Gro(\Var_F)\to\Gro^{\Bitt}(\Var_F)$ that maps $[X]$ to the element \eqref{eq:Bittproof} for every separated smooth $F$-scheme $X$ of finite type. This morphism is inverse to the morphism $\Gro^{\Bitt}(\Var_F)\to \Gro(\Var_F)$.
\end{proof}
\begin{remark}
Bittner's presentation remains valid we replace ``proper'' by ``projective'', by the same proof. 
\end{remark}

\begin{theorem}[Larsen \& Lunts 2003]\label{thm:LL}
Let $F$ be a field of characteristic zero. Let $\SB_F$ be the set of stable birational equivalence classes $\{X\}_{\sbir}$ of integral\footnote{We follow the convention that integral schemes are non-empty.} $F$-schemes $X$ of finite type, and let $\Z[\SB_F]$ be the free abelian group on the set $\SB_F$. For every $F$-scheme $Y$ of finite type, we set $$\{Y\}_{\sbir}=\{Y_1\}_{\sbir}+\ldots+\{Y_r\}_{\sbir}$$ in $\Z[\SB_F]$, where $Y_1,\ldots,Y_r$ are the irreducible components of $Y$. In particular, $\{\emptyset\}_{\sbir}=0$. We endow $\Z[\SB_F]$ with the unique ring structure such that 
$$\{Y\}_{\sbir}\cdot \{Y'\}_{\sbir}=\{Y\times_F Y'\}_{\sbir}$$
for all $F$-schemes $Y$ and $Y'$ of finite type.

Then there exists a unique group morphism 
$$\sbir\colon \Gro(\Var_F)\to \Z[\SB_F]$$ that maps $[X]$ to $\{X\}_{\sbir}$ for every smooth and proper $F$-scheme $X$. The morphism $\sbir$ is a surjective ring morphism, and its kernel is the ideal in  $\Gro(\Var_F)$ generated by $\LL$.
\end{theorem}
\begin{proof}
This is a combination of Theorem 2.3 and Proposition 2.7 in \cite{larsen-lunts}.
 The existence of $\sbir$ is an immediate consequence of Theorem \ref{thm:Bitt}\footnote{Theorem \ref{thm:LL} slightly predates Theorem \ref{thm:Bitt}, and in \cite{larsen-lunts}, the existence of $\sbir$ was deduced directly  from the Weak Factorization Theorem.}. Indeed, in the setting of \eqref{eq:blup} (and excluding the trivial case $Y=X$), $\mathrm{Bl}_Y X$ is birational to $X$, and $E$ is birational to $Y\times_F \PP^{\dim(X)-\dim(Y)-1}_F$. It is obvious that $\sbir$ is unique, and that it is a ring morphism.  
 
The morphism $\sbir$ maps $\LL=[\Pro^1_F]-[\Spec F]$ to $0$, because $\Spec F$ is stably birational to $\Pro^1_F$.  Thus $\sbir$ induces a ring morphism 
 $$\overline{\sbir}\colon \Gro(\Var_F)/\LL \Gro(\Var_F)\to \Z[\SB_F].$$
 We prove that this is an isomorphism by constructing an inverse. 
By Hironaka's resolution of singularities, every class in $\SB_F$ has a representative $X$ that is a connected smooth proper $F$-scheme.  
 For every $m\geq 0$, we have $$[X\times_F \Pro^m_F]-[X]=[X](\LL+\LL^2+\ldots+\LL^m)$$ in $\Gro(\Var_F)$ by the scissor relations. Thus $[X\times_F \Pro^m_F]$ and $[X]$ are congruent modulo $\LL$. Moreover, the congruence class of $[X\times_F \Pro^m_F]$ modulo $\LL$ is independent under blow-ups of smooth closed subschemes of $X\times_F \Pro^m_F$, because the exceptional divisor of such a blow-up is a projective bundle over the center. Now it follows from the Weak Factorization Theorem that the class of $X$ in $\Gro(\Var_F)/\LL\Gro(\Var_F)$ only depends on the stable birational equivalence class of $X$. 
 This yields a ring morphism 
 $$\Z[\SB_F]\to \Gro(\Var_F)/\LL\Gro(\Var_F)$$ that is inverse to $\overline{\sbir}$. 
\end{proof}

Beware that $\sbir([X])$ is usually different from $\{X\}_{\sbir}$ when $X$ is not smooth and proper. For instance, if $X$ is a nodal cubic in $\Pro^2_F$, then it follows from the scissor relations that $[X]=\LL$ in $\Gro(\Var_F)$. Thus $\sbir([X])=0$.

\begin{cor}\label{cor:LL}
Let $F$ be a field of characteristic zero, and let $X$ and $X'$ be smooth and proper $F$-schemes. Then $X$ and $X'$ are stably birational if and only if $[X]\equiv [X']$ modulo $\LL$ in $\Gro(\Var_F)$. 

In particular, $[X]\equiv c$ modulo $\LL$ for some integer $c$ if and only if every connected component of $X$ is stably rational; in that case, $c$ is the number of connected components of $X$.
\end{cor}
\begin{proof}
By the scissor relations in the Grothendieck ring, we can write 
$$[X]=[X_1]+\cdots +[X_r]$$ where $X_1,\ldots,X_r$ are the connected components of $X$. Now the result follows immediately from Theorem \ref{thm:LL}.
\end{proof}

Corollary \ref{cor:LL} shows that the Grothendieck ring of varieties detects the stable birational type of smooth and proper schemes over fields of characteristic zero. The analogous question in positive characteristic is open. We are not aware of any example of a pair of smooth and proper schemes $X,\,X'$ over a field $F$ such that $[X]=[X']$ in $\Gro(\Var_F)$ and such that $X$ and $X'$ are not birational. Thus, to the best of our knowledge, it remains an open question whether the Grothendieck ring detects birational types of smooth and proper schemes (even in characteristic zero). To overcome this problem, we will introduce in Section \ref{ss:graded} a finer variant of the Grothendieck ring of varieties, graded by dimension.

\begin{remark}\label{rem:borisov}
 Corollary \ref{cor:LL} is false without the assumption that $X$ and $X'$ are smooth and proper.
  Borisov has constructed an example of two complex Calabi-Yau threefolds $Z$ and $Z'$ that are not birational and such that $[Z\times\A^6_{\C}]=[Z'\times\A^6_{\C}]$ in $\Gro(\Var_{\C})$ (see \cite{borisov} and the subsequent refinement in \cite{martin} and \cite{motbook}). Thus $Z\times_{\C} \A^6_{\C}$ and $Z'\times_{\C}\A^6_{\C}$ are smooth complex varieties that are not stably birational and that define the same class in the Grothendieck ring. 
\end{remark}

\subsection{The graded Grothendieck ring}\label{ss:graded}
Let $F$ be a field, and let $d$ be a non-negative integer. 
 We define the Grothendieck group $\Gro(\Var_F^{\leqslant d})$ of $F$-varieties of dimension at most $d$ to be the abelian group with the following presentation.
 \begin{itemize}
 \item {\em Generators:} isomorphism classes $[X]_d$ of $F$-schemes $X$ of finite type and of dimension at most $d$;
 \item {\em Relations:} whenever $X$ is an $F$-scheme of finite type and of dimension at most $d$, and $Y$ is a closed subscheme of $X$, then $$[X]_d=[Y]_d+[X\setminus Y]_d.$$
 \end{itemize}

 The Grothendieck group of varieties graded by dimension is the graded abelian group
 $$\Grodim(\Vardim_F)=\bigoplus_{d\geq 0}\Gro(\Var_F^{\leqslant d}).$$
It has a unique structure of a graded ring such that 
$$[X]_d\cdot [X']_e=[X\times_F X']_{d+e}$$ for all $F$-schemes $X$ and $X'$ of finite type and of dimensions at most $d$ and $e$, respectively. The identity element for the ring multiplication is the element $1=[\Spec F]_0$.
 With a slight abuse of notation, we will also use the symbol $\LL$ to denote the class $[\A^1_F]_1$ of the affine line in degree $1$. We set $\tau=[\Spec F]_1$, the class of the point in $\Gro(\Var_F^{\leqslant 1})$. The multiplicative action of $\tau$ on $\Grodim(\Vardim_F)$ shifts the degree: for every $F$-scheme $X$ of finite type of dimension at most $d$, and for every integer $e\geq 0$, we have $\tau^e [X]_d=[X]_{d+e}$.

The graded Grothendieck ring is related to the usual Grothendieck ring $\Gro(\Var_F)$ in the following way.

\begin{prop}\label{prop:grogro}
There exists a unique ring morphism 
$$\Grodim(\Vardim_F)\to \Gro(\Var_F)$$ that maps $[X]_d$ to $[X]$, for every non-negative integer $d$ and every $F$-scheme $X$ of finite type and of dimension at most $d$. In particular, it maps $\LL=[\A^1_F]_1$ to $\LL=[\A^1_F]$.
 This morphism is surjective, and its kernel is the ideal generated by $\tau-1$.
\end{prop}
\begin{proof}
It is clear from the definitions of the Grothendieck rings that there is a unique ring morphism mapping $[X]_d$ to $[X]$. It is also obvious that it is surjective, and that its kernel contains $\tau-1$. Thus this morphism factors through a ring morphism 
$$\Grodim(\Vardim_F)/(\tau-1)\to \Gro(\Var_F).$$
This is an isomorphism: its inverse maps $[X]$ to $[X]_d$ for every $F$-scheme $X$ of finite type, where $d$ is any integer such that $d\geq \dim(X)$.
\end{proof}

\subsection{Birational types}
Let $F$ be a field, and let $d$ be a non-negative integer. We denote by $\Bir^d_F$ the set of birational equivalence classes of integral $F$-schemes $X$ of finite type of dimension $d$. The equivalence class of $X$ will be denoted by $\{X\}_{\bir}$. Let $\Z[\Bir_F^d]$ be the free abelian group on the set $\Bir_F^d$. We set $\Bir_F=\cup_{d\geq 0}\Bir_F^d$. This is the set of birational equivalence classes of integral $F$-schemes of finite type.
 We also introduce the graded abelian group 
 $$\Z[\Bir_F]=\bigoplus_{d\geq 0}\Z[\Bir_F^d].$$ It has a unique structure of a graded ring 
 such that $$\{X\}_{\bir}\cdot \{X'\}_{\bir}=\sum_{i=1}^r \{C_i\}_{\bir}$$
 for all integral $F$-schemes $X$ and $X'$ of finite type, where $C_1,\ldots,C_r$ are the irreducible components of $X\times_F X'$ (endowed with their induced reduced structures).
 For a field $F$ of characteristic zero, this graded ring was introduced by Kontsevich and Tschinkel in Section 2 of \cite{KT}; there it was called the {\em Burnside ring} of $F$.

\begin{prop}\label{prop:grobir}
There exists a unique morphism of graded rings 
$$\bir\colon \Grodim(\Vardim_F)\to \Z[\Bir_F]$$ 
such that $\bir([X]_d)=\{X\}_{\bir}$ when $X$ is an integral $F$-scheme of finite type of dimension $d$, and $\bir([X]_d)=0$ whenever $X$ is an $F$-scheme of finite type of dimension at most $d-1$. This morphism is surjective, and its kernel is the ideal generated by $\tau$.
\end{prop}
\begin{proof}
Let $d$ be a non-negative integer, and let $X$ be an $F$-scheme of finite type and of dimension at most $d$. 
Let $X_1,\ldots,X_r$ be the irreducible components of $X$ of dimension $d$, endowed with their induced reduced structures.  
 We set $$\bir([X]_d)=\{X_1\}_{\bir}+\ldots +\{X_r\}_{\bir}.$$
This definition respects the scissor relations in $\Gro(\Var^{\leqslant d}_F)$, and therefore induces a morphism of graded groups 
$$\bir\colon \Grodim(\Vardim_F)\to \Z[\Bir_F].$$ 
It follows immediately from the definitions of the ring multiplications on the source and the target that $\bir$ is a morphism of graded rings. The uniqueness property in the statement follows from the fact that, by the scissor relations, the element  
$$[X]_d-[X_1]_d-\ldots -[X_r]_d$$ in $\Gro(\Var_F^{\leqslant d})$ can be written as a linear combination of classes $[Y]_d$ where $Y$ is an $F$-scheme of finite type of dimension at most $d-1$. 
 
It is obvious that the morphism $\bir$ is surjective, and that its kernel contains the ideal generated by $\tau$. Thus it induces a ring morphism 
 $$\Grodim(\Vardim_F)/(\tau)\to \Z[\Bir_F].$$ 
This is an isomorphism: its inverse maps $\{X\}_{\bir}$ to $[X]_{\dim(X)}$, for every 
 integral $F$-scheme $X$ of finite type. 
\end{proof}

\begin{cor}\label{cor:grobir}
If $X$ and $X'$ are reduced $F$-schemes of finite type of pure dimension $d$, then $X$ and $X'$ are birational if and only if  
$[X]_d\equiv [X']_d$ modulo $\tau$ in $\Grodim(\Vardim_F)$.  
\end{cor}
\begin{proof}
If $X_1,\ldots,X_r$ are the irreducible components of $X$, then 
$$\bir([X]_d)=\{X_1\}_{\bir}+\ldots+\{X_r\}_{\bir}$$ in $\Z[\Bir_F]$, and the analogous property holds for $X'$. Thus $X$ and $X'$ are birational if and only if $\bir([X]_d)=\bir([X']_d)$. By Proposition \ref{prop:grobir}, this is equivalent to the property that
$[X]_d\equiv [X']_d$ modulo $\tau$ in $\Grodim(\Vardim_F)$.  
\end{proof}

Corollary \ref{cor:grobir} tells us that the graded version of the Grothendieck ring detects birational types; this is its main advantage over the classical Grothendieck ring of varieties. 

\begin{exam}
Let $Z$ and $Z'$ be the complex Calabi-Yau threefolds from Borisov's example (see Remark \ref{rem:borisov}). Then $Z\times_{\C}\A^6_{\C}$ and $Z'\times_{\C}\A^6_{\C}$ define the same class in $\Gro(\Var_{\C})$. However, since these varieties are not birational, it follows from Corollary \ref{cor:grobir} that their classes in $\Gro(\Var^{\leqslant 9}_{\C})$ are distinct. Thus the difference of these two classes lies in the kernel of the map
$\Grodim(\Vardim_{\C})\to \Gro(\Var_{\C})$ defined in Proposition \ref{prop:grogro}. 
\end{exam}

\begin{remark}
A different manifestation of the isomorphism in Proposition \ref{prop:grobir} also appears in Proposition 2.2 of \cite{KrTsch}.
\end{remark}

\subsection{A refinement of Bittner's presentation}\label{ss:dimBitt}
We will now establish an analog of Bittner's presentation (Theorem \ref{thm:Bitt}) for the graded Grothendieck ring $\Grodim(\Vardim_F)$, where $F$ is a field of characteristic zero. 

For every non-negative integer $d$, we define an abelian group $\Gro^{\Bitt}(\Var^{\leqslant d}_F)$ by means of the following presentation.
\begin{itemize}
\item {\em Generators:} isomorphism classes $[X]_d^{\Bitt}$ of connected smooth and proper $F$-schemes $X$ of dimension at most $d$;
\item {\em Relations:} $[\emptyset]^{\Bitt}_d=0$, and, whenever $X$ is a connected smooth and proper $F$-scheme of dimension at most $d$, and $Y$ is a connected smooth closed subscheme of $X$, then 
\begin{equation}\label{eq:dimblup}
[\mathrm{Bl}_YX]_d^{\Bitt}-[E]_d^{\Bitt}=[X]_d^{\Bitt}-[Y]_d^{\Bitt}
\end{equation}
where $\mathrm{Bl}_YX$ denotes the blow-up of $X$ along $Y$, and $E$ is the exceptional divisor.
\end{itemize}

\begin{theorem}\label{thm:dimBitt}
Let $F$ be a field of characteristic zero. 
For every non-negative integer $d$, there exists a unique group morphism 
$$\Gro^{\Bitt}(\Var^{\leqslant d}_F)\to \Gro(\Var^{\leqslant d}_F)$$
that maps $[X]_d^{\Bitt}$ to $[X]_d$, for every connected smooth and proper $F$-scheme $X$ of dimension at most $d$. 
 This morphism is an isomorphism.
\end{theorem}
\begin{proof}
One can simply copy the proof of Theorem 3.1 in \cite{bittner}: all the schemes involved in the 
argument have dimension at most $d$ (see Theorem \ref{thm:Bitt} for a sketch of the proof).
\end{proof}

\subsection{A refinement of the theorem of Larsen \& Lunts}
We can also prove a refinement of the theorem of Larsen and Lunts (Theorem \ref{thm:LL}) in the graded setting. This refinement will not be used in the remainder of the paper, because we already know from Corollary \ref{cor:grobir} that the graded Grothendieck ring detects birational types.

Let $F$ be a field, and let $d$ be a non-negative integer. Let $X$ and $X'$ be irreducible $F$-schemes of finite type such that $\dim(X)\leq d$ and $\dim(X')\leq d$.
 We say that $X$ and $X'$ are {\em $d$-stably birational} if $X\times_F \Pro^{d-\dim(X)-1}_F$ is birational to  $X'\times_F \Pro^{d-\dim(X')-1}_F$. Here and below, we use the convention that $\Pro^{-1}_F=\Spec F$. See Remark \ref{rem:LL} for a comment on the appearance of the term $-1$ in the dimensions of the projective spaces. Note that, if $d=\dim(X)=\dim(X')$ or $d-1=\dim(X)=\dim(X')$, then $X$ and $X'$ are $d$-stably birational if and only if they are birational.

 We denote by $\SB^{\leqslant d}_F$ the set of $d$-stable birational equivalence classes of  integral $F$-schemes $X$ of finite type of dimension at most $d$. 
 The equivalence class of $X$ will be denoted by $\{X\}_{\sbir,d}$. 
 Let $\Z[\SB^{\leqslant d}_F]$ be the free abelian group on the set $\SB^{\leqslant d}_F$.
  For every $F$-scheme $Y$ of finite type of dimension at most $d$, we set 
  $$\{Y\}_{\sbir,d}=\{Y_1\}_{\sbir,d}+\ldots+\{Y_r\}_{\sbir,d}$$
  in $\Z[\SB^{\leqslant d}_F]$, where $Y_1,\ldots,Y_r$ are the irreducible components of $Y$ (with their induced reduced structures). In particular, $\{\emptyset\}_{\sbir,d}=0$. 
 We consider the graded abelian group
$$\Z[\SBdim{F}]=\bigoplus_{d\geq 0}\Z[\SB^{\leqslant d}_F].$$
and endow it with the unique structure of a graded ring 
 such that $$\{Y\}_{\sbir,d}\cdot \{Y'\}_{\sbir,e}=\{Y\times_F Y'\}_{\sbir,d+e}$$
whenever $d$ and $e$ are non-negative integers, and $Y$ and $Y'$ are $F$-schemes of finite type of dimensions at most $d$ and $e$, respectively.

\begin{lemma}\label{lemm:LL}
For all integers $d$ and $n$ such that $d\geq n\geq 0$ and $(d,n)\neq (1,1)$, we have 
$[\Pro^n_F]_d\equiv \tau^d$ modulo $[\A^1_F]_2$ in $\Grodim(\Vardim_F)$.
\end{lemma}
\begin{proof}
 If $n=0$ then the assertion is trivial; thus we may assume that $n\geq 1$ and $d\geq 2$.  
The scissor relations imply that 
$$[\Pro^n_F]_d=[\Spec k]_d+[\A^1_F]_d+\ldots +[\A^n_F]_d.$$
Thus we can write 
$$[\Pro^n_F]_d-\tau^d=[\A^1_F]_2(\tau^{d-2}+\ldots +[\A^{n-1}_F]_{d-2})$$
in $\Grodim(\Vardim_F)$.
\end{proof}
\begin{theorem}\label{thm:dimLL}
Let $F$ be a field of characteristic zero.  Then there exists a unique morphism of graded rings  
$$\sbir_{\dim}\colon \Grodim(\Vardim_F)\to \Z[\SBdim{F}]$$
such that $\sbir_{\dim}([X]_d)=\{X\}_{\sbir,d}$ for every non-negative integer $d$ and every smooth and proper $F$-scheme $X$ of dimension at most $d$.

 The morphism $\sbir_{\dim}$ is surjective, and its kernel is the ideal generated by $\tau\LL=[\A^1_F]_2$.  
\end{theorem}
\begin{proof}
The existence and uniqueness of the morphism $\sbir_{\dim}$ follow immediately from Theorem \ref{thm:dimBitt}; note that, in the blow-up relations \eqref{eq:dimblup} (and excluding the trivial case $Y=X$), the exceptional divisor $E$ is birational to $Y\times_F \Pro^{\dim(X)-\dim(Y)-1}$, so that $E$ and $Y$ are $\dim(X)$-stably birational. The kernel of $\sbir_{\dim}$ contains $[\A^1_F]_2=[\Pro^1_F]_2-\tau^2$, because 
$\Pro^1_F$ and $\Spec F$ are $2$-stably birational. Thus $\sbir_{\dim}$ induces a morphism of graded rings 
$$\overline{\sbir}_{\dim}\colon \Grodim(\Vardim_F)/([\A^1_F]_2)\to \Z[\SBdim{F}].$$
We will show that this is an isomorphism by constructing its inverse. 

 For every smooth and proper $F$-scheme $Y$ and every integer $d\geq \dim(Y)$,
  the residue class of $[Y]_d$ in $\Grodim(\Vardim_F)/([\A^1_F]_2)$ is invariant under any blow-up of a smooth strict closed subscheme $Z$ in $Y$. Indeed, we may assume that the codimension of $Z$ in $Y$ is at least $2$; if we denote by $E$ the exceptional divisor 
  in the blow-up $\mathrm{Bl}_Z Y$ of $Y$ at $Z$, then the scissor relations imply that 
$$\begin{array}{lcl}
[\mathrm{Bl}_Z Y]_d-[Y]_d&=&[E]_d-[Z]_d
\\ &=& [Z]_{\dim(Z)}([\Pro^{\dim(Y)-\dim(Z)-1}_F]_{d-\dim(Z)}-\tau^{d-\dim(Z)})
\end{array}$$ 
in $\Grodim(\Vardim_F)$, and this element is divisible by $[\A^1_F]_2$ by Lemma \ref{lemm:LL}.
    Therefore, by the Weak Factorization Theorem \cite{WF}, the residue class of $[Y]_d$ in $\Grodim(\Vardim_F)/([\A^1_F]_2)$ only depends on the birational equivalence class of $Y$.
  
Now let $X$ be a non-empty connected smooth and proper $F$-scheme, and let $d$ be an integer such that $d\geq \dim(X)$.  
   Then the residue class of $[X\times_F \Pro^{d-\dim(X)-1}_F]_d$ in  $\Grodim(\Vardim_F)/([\A^1_F]_2)$ only depends on the birational equivalence class of $X\times_F \Pro^{d-\dim(X)-1}_F$; in other words, it only depends on the
    $d$-stable birational equivalence class of $X$. 
    But we also have $[X\times_F \Pro^{d-\dim(X)-1}_F]_d=[X]_d$ in 
     $\Grodim(\Vardim_F)/([\A^1_F]_2)$ by Lemma \ref{lemm:LL}.
      Thus we obtain a morphism of graded abelian groups 
  $$\Z[\SBdim{F}]\to \Grodim(\Vardim_F)/([\A^1_F]_2)\colon \{X\}_{\sbir,d}\mapsto [X]_d$$
that is inverse to $\overline{\sbir}_{\dim}$.
\end{proof}

\begin{cor}\label{cor:dimLL}
Let $F$ be a field of characteristic zero. 
\begin{enumerate}
\item Let $X$ and $X'$ be connected smooth and proper $F$-schemes.
Let $d$ be an integer such that $d\geq \dim(X)$ and $d\geq \dim(X')$. Then $X$ and $X'$ are $d$-stably birational if and only if $[X]_d\equiv [X']_d$ modulo $[\A^1_F]_2$ in $\Grodim(\Vardim_F)$. 

\item Let $X$ be a smooth and proper $F$-scheme. Then $[X]\equiv c$ modulo $[\A^1_F]_2$ for some integer $c$ if and only if, for every connected component $C$ of $X$, the scheme $C\times_F \Pro^{d-\dim(C)-1}_F$ is rational; in that case, $c$ is the number of connected components of $X$.
\end{enumerate}
\end{cor}
\begin{proof}
This follows immediately from Theorem \ref{thm:dimLL}.
\end{proof}

\begin{remark}\label{rem:LL}
One can also formulate a  weaker version of Theorem \ref{thm:dimLL}, replacing the exponents $d-\dim(X)-1$ and $d-\dim(X')-1$ in the definition of $d$-stable birational equivalence by $d-\dim(X)$ and $d-\dim(X')$. With that definition, the morphism $\sbir_{\dim}$ constructed in Theorem \ref{thm:dimLL} has kernel generated by $\LL=[\A^1_F]_1$, rather than $[\A^1_F]_2$.    
\end{remark}

\section{Dimensional refinement of the motivic volume}\label{sec:motvol}
\subsection{The motivic volume}
Let $k$ be a field of characteristic zero, and set 
$$K=\bigcup_{n>0}k\llpar t^{1/n}\rrpar,\quad R=\bigcup_{n>0}k\llbr t^{1/n}\rrbr.$$
The field $K$ is a henselian valued field with valuation ring $R$ with respect to the $t$-adic valuation 
$$\mathrm{ord}_t\colon K^{\times}\to \Q.$$ If $k$ is algebraically closed then $K$ is an algebraic closure of the Laurent series field $k\llpar t\rrpar$, but we do not make this assumption.

 In \cite{NiSh}, Shinder and the first-named author constructed a ring morphism 
 $$\Vol\colon \Gro(\Var_K)\to \Gro(\Var_k),$$ called the {\em motivic volume}. It maps $[\A^1_K]$ to $[\A^1_k]$ and has the property that for every smooth and proper $R$-scheme $\cX$, one has $\Vol([\cX_K])=[\cX_k]$. This ring morphism can be viewed as a refinement of the motivic nearby fiber of Denef and Loeser \cite{DL}, where the refinement consists of the fact that we do not need to invert $[\A^1_k]$ in the target; this is crucial for applications to rationality problems. If $k$ is algebraically closed, then the existence of the morphism $\Vol$ follows immediately from the work of Hrushovski and Kazhdan \cite{HK}. 
 
 The fact that $\Vol([\A^1_K])=[\A^1_k]$ implies that the morphism $\Vol$ factors through a ring morphism 
 $\Gro(\Var_K)/\LL\to \Gro(\Var_k)/\LL$ which, by Theorem \ref{thm:LL}, we can identify with a ring morphism $$\Volsb\colon \Z[\SB_K]\to \Z[\SB_k].$$ If $\cX$ is a connected smooth and proper $R$-scheme, then $\Volsb$ maps the stable birational equivalence class of $\cX_K$ to that of $\cX_k$. Since $\Vol([\Spec K])=[\Spec k]$, the morphism $\Volsb$ maps the class of stably rational $K$-varieties to the class of stably rational $k$-varieties. It then follows easily that stable rationality of geometric fibers specializes in smooth and proper families, which was one of the main results of \cite{NiSh}. 

 The main purpose of the present article is to upgrade the motivic volume to a morphism of graded rings 
 $$\Vol\colon \Grodim(\Vardim_K)\to \Grodim(\Vardim_k)$$ that maps $[\Spec K]_1$ to $[\Spec k]_1$ and fits into a a commutative diagram 
$$\begin{CD} 
\Grodim(\Vardim_K)@>\Vol>> \Grodim(\Vardim_k)
\\ @VVV @VVV
\\ \Gro(\Var_K)@>>\Vol>\Gro(\Var_k) 
\end{CD}$$
where the vertical morphims are the forgetful maps from Proposition \ref{prop:grogro}. The refined volume induces a morphism of graded rings $$\Grodim(\Vardim_K)/([\Spec K]_1)\to \Grodim(\Vardim_k)/([\Spec k]_1)$$ which, by Proposition \ref{prop:grobir}, can be identified with a morphism of graded rings 
$$\Volbir\colon \Z[\Bir_K]\to \Z[\Bir_k].$$ This is precisely the birational specialization morphism from \cite{KT}.

 The construction of the motivic volume in \cite{NiSh} was phrased in the language of logarithmic geometry. For readers familiar with that language it will be an easy exercise to check that all the arguments in \cite{NiSh} still apply to the dimensional refinement of the Grothendieck ring. Therefore, we decided not to reproduce the proofs here, but to explain the main properties in a more explicit way, avoiding as much as possible the language of logarithmic schemes. We hope that this will make the theory more user-friendly. It also makes our formula for the motivic volume a bit more general, because our class of strictly toroidal models (defined below) is slightly more general than the log smooth models that were considered in \cite{NiSh}, since one does not need to assume that the log structure is defined globally on the model. This will be explained in more detail in the proof of Theorem \ref{theo:MV}.
 
 There are two further differences in presentation compared to Appendix A of  \cite{NiSh}: we directly work over $K$, rather than $k\llpar t\rrpar$, which means that we ignore the monodromy action of the profinite group scheme $\hat{\mu}$ of roots of unity over $k$ (we will come back to this monodromy action in Section \ref{sec:monodromy}). Second, while the formulas in \cite{NiSh} were stated in terms of {\em open} strata in the special fibers of log smooth models, we will also express our formulas in terms of {\em closed} strata. This is more convenient for applications to rationality questions, and for a comparison with the invariant defined in \cite{KT}. Passing between open and closed strata can always be done by means of basic inclusion-exclusion arguments (see Lemma \ref{lemm:incexc}).
 
 Finally, let us also point out a typo in \cite{NiSh}: in the formulas in Theorem A.3.9 and Propostion A.4.1, the factor $(\LL-1)^{r_v(\sigma)}$ should be $(\LL-1)^{r_v(\sigma)-1}$, like in the expression for $\Vol_K(X)$ in the middle of page 407.

\subsection{Strictly toroidal models}
We say that a monoid $M$ is {\em toric} if it is sharp\footnote{A monoid $M$ is called {\em sharp} if 
the only invertible element of $M$ is the identity.}, finitely generated, integral, saturated and torsion free. Toric monoids are precisely those monoids that are isomorphic to the monoid of lattice points in a strictly convex rational polyhedral cone $\sigma$ in $\R^d$, for some positive integer $d$. There is an intrinsic definition of the faces of $M$ and their dimensions; if  $M=\sigma\cap \Z^d$ then the faces of $M$ are the submonoids of the form $\tau\cap \Z^d$ with $\tau$ a face of $\sigma$, and the dimension of $\tau\cap \Z^d$ equals the dimension of $\tau$.
 For any ring $A$, we denote by $A[M]$ the monoid $A$-algebra associated with $M$. Its elements are the finite $A$-linear combinations of monomials $\chi^m$ with $m$ in $M$.

Let $\cX$ be a flat separated $R$-scheme of finite type, and let $x$ be a point of the special fiber $\cX_k$. We say that $\cX$ is {\em strictly toroidal} at $x$ if there exist a toric monoid $M$, an open neighbourhood $\cU$ of $x$ in $\cX$, and a smooth morphism of $R$-schemes  
\begin{equation}\label{eq:tor}
\cU\to \Spec R[M]/(\chi^m-t^q) 
\end{equation} where $q$ is a positive rational number, and $m$ is an element of $M$ such that $k[M]/(\chi^m)$ is reduced. This implies in particular that $\cX$ is normal at $x$ and $\cX_k$ is reduced at $x$.
% The {\em vertical rank} of $\cX$ at $x$ is the dimension of the smallest face of $M$ that contains $m$. It is independent of the choice of $M$ and $m$, and we will denote it by $r_v(x)$.
 We say that $\cX$ is strictly toroidal if it is so at every point $x$ of $\cX_k$. 

\begin{example}\label{exam:snctor}
 A flat separated $R$-scheme $\cX$ of finite type is called {\em strictly semi-stable} if, Zariski-locally, it admits a smooth morphism to a scheme of the form 
\begin{equation}\label{eq:sstable}
\Spec R[z_1,\ldots,z_n]/(z_1\cdot \ldots \cdot z_n-t^q)
\end{equation}
 for some $n\geq 1$ and some positive rational number $q$. Then $\cX$ is also strictly toroidal (take $M=\N^{n}$ and set $m=(1,\ldots,1)$). Every smooth and proper $K$-scheme $X$ has a strictly semi-stable proper $R$-model:
 we can descend $X$ to $k\llpar t^{1/d}\rrpar$ for some $d>0$; by Hironaka's resolution of singularities, we can then find a proper regular $k\llbr t^{1/d}\rrbr$-model $\cY$ whose special fiber is  a divisor with strict normal crossings (not necessarily reduced). An elementary local calculation now shows that the normalization of $\cY\times_{k\llbr t^{1/d}\rrbr}R$ is strictly semi-stable proper $R$-model of $X$.
\end{example}

\begin{example}\label{exam:tor}
 The class of strictly toroidal $R$-models is more flexible than that of strictly semi-stable models. This is useful in applications to rationality problems, as one is allowed to bypass an explicit resolution of singularities over $k\llbr t\rrbr$. For instance, one can skip the construction in Lemma 2.2 of Shinder's paper \cite{shinder}; other applications are given in Section \ref{ss:exam} and in \cite{NO}.
 
 Let $f_0\in k[z_0,\ldots,z_n]$ be a general homogeneous polynomial of positive  degree $d_0$. Let $f_1,\ldots,f_r\in k[z_0,\ldots,z_n]$ be general homogeneous polynomials of positive degrees $d_1,\ldots,d_r$, respectively, such that $d_1+\ldots +d_r=d_0$. Then 
$$\cX=\mathrm{Proj}\,R[z_0,\ldots,z_n]/(tf_0-f_1\cdot \ldots \cdot f_r)$$ is strictly toroidal, but not strictly semi-stable if $r\geq 2$ and $n\geq 3$: locally around the singular points of $\cX_k$ where $f_0$ vanishes, there is no smooth morphism to a scheme of the form \eqref{eq:sstable}.

 To see that $\cX$ is strictly toroidal, let $x$ be a point on $\cX_k$, and let $I$ be the set of indices $i\in \{0,\ldots,r\}$ such that $f_i(x)=0$. After a permuation of the coordinates $z_j$, we may assume that $z_0(x)\neq 0$. Since the polynomials $f_i$ are general, the regular functions $f_i/z_i^{d_i}$ with $i\in I$ form a part of a regular system of local parameters in $\mathcal{O}_{\mathbb{P}^n_{R},x}$. Therefore, there exists an open neighbourhood of $x$ in $\cX$ that admits a smooth morphism to the $R$-scheme 
 $$\cY=\Spec R[y_i\,|\,i\in I]/(tw - \prod_{i\in I\setminus \{0\}}y_i)$$
 where $w=y_0$ if $0\in I$ and $w=1$ otherwise. In the latter case, $\cY$, and therefore $\cX$, are strictly semi-stable and, thus, also strictly toroidal. In the former case, $\cY$ is isomorphic to the $R$-scheme 
 $$\Spec R[M]/(t-\chi^m)$$ where  $M$ is the quotient of the monoid $\N^I\times \N$ by the congruence relation generated by $$(e_0,1)\sim (\sum_{i\in I\setminus \{0\}}e_i,0),$$
 where $(e_i)_{i\in I}$ is the standard basis of $\N^I$, and $m$ is the residue class of $(0,1)\in \N^I\times \N$.
\end{example} 
 
 If $\cX$ is a strictly toroidal $R$-scheme, then a {\em stratum} $E$ of $\cX_k$ is a connected component of the intersection of a non-empty set of irreducible components of $\cX_k$. We denote by $\codim(E)$ the codimension of $E$ in $\cX_k$, and by $\mathcal{S}(\cX)$ the set of strata in $\cX_k$. 
%The vertical rank of $\cX$ at $E$ is defined to be the vertical rank at the generic point of $E$; we denote it by $r_v(E)$.  
 The {\em interior} $E^o$ of a stratum $E$ is the complement in $E$ of the union of strictly smaller strata. This interior $E^o$ is a connected smooth separated $k$-scheme of finite type, but $E$ may have singularities along the boundary $E\setminus E^o$. All of these singularities are strictly toroidal, in the sense that $E$ admits Zariski-locally an \'etale morphism to a toric $k$-variety. 
 
 We also attach to $E$ an element  in $\Grodim(\Vardim_k)$ given by the formula 
 $$P(E)=\sum_{E'\supset E}[\mathbb{G}^{\codim(E')}_{m,k}]_{\codim(E)}$$ where the sum is taken over all the strata $E'$ in $\cX_k$ that contain $E$. If $E$ is contained in precisely $\codim(E)+1$ irreducible components of $\cX_k$, then $P(E)=[\mathbb{P}^{\codim(E)}_k]_{\codim(E)}$ because the sum in the definition corresponds to the partition of  $\mathbb{P}^{\codim(E)}_k$ into torus orbits. This happens for instance when the toroidal structure of $\cX$ at the generic point of $E$ is simplicial. In general, one can write $P(E)$ as the class of a proper toric $k$-variety of dimension $\codim(E)$, which can be computed explicitly from the toroidal structure (that is, the monoid $M$) at the generic point of $E$. The only thing that matters for our applications is that  the image of $P(E)$ in $\Z[\Bir_k]$ with respect to the morphism in Proposition \ref{prop:grobir} is always equal to $\{\mathbb{P}^{\codim(E)}_k\}_{\bir}$; this follows immediately from the definition of $P(E)$.

 %The vertical rank of $\cX$ at $x$ equals $r_v(E)$ for all $x$ in $E^o$, but $r_v(x)>r_v(E)$ when $x$ lies in the boundary $E\setminus E^o$.

\subsection{Construction of the motivic volume} 

\begin{lemma}\label{lemm:incexc}
If $\cX$ is a strictly toroidal $R$-scheme, then for every $e\geq \dim(\cX_k)$,
$$\sum_{E\in \mathcal{S}(\cX)}(-1)^{\codim(E)}[E^o\times_k \mathbb{G}^{\codim(E)}_{m,k}]_e 
= \sum_{E\in \mathcal{S}(\cX)}(-1)^{\codim(E)}[E]_{e-\codim(E)}P(E). $$
\end{lemma}
\begin{proof}
 Every stratum $E$ in $\cX_k$ is the disjoint union of its open substrata $(E')^o$, so that  $$[E]_{e-\codim(E)}=\sum_{E'\subset E}[(E')^o]_{e-\codim(E)}$$ in $\Grodim(\Vardim_k)$. Thus we can write the right hand side of the equality in the statement as 
\begin{eqnarray*}
 & &\sum_{E \in \mathcal{S}(\cX)}\sum_{E'\subset E}(-1)^{\codim(E)}[(E')^o]_{e-\codim(E)}P(E)
 \\ &=&\sum_{E'\in \mathcal{S}(\cX)}\left([(E')^o]_{e-\codim(E')}\sum_{E\supset E'}(-1)^{\codim(E)}P(E)\tau^{\dim(E)-\dim(E')} \right).
 \end{eqnarray*} 
By the definition of the element $P(E)$, we have for every $E'$ in $\mathcal{S}(\cX)$ that 
\begin{eqnarray*}
& & \sum_{E\supset E'}(-1)^{\codim(E)}P(E)\tau^{\dim(E)-\dim(E')}
\\ &=&
\sum_{E\supset E'}\left((-1)^{\codim(E)}\sum_{E''\supset E}[\mathbb{G}_{m,k}^{\codim(E'')}]_{e-\dim(E')} \right)
\\ &=& \sum_{E''\supset E'}\left([\mathbb{G}_{m,k}^{\codim(E'')}]_{e-\dim(E')}\sum_{E'\subset E \subset E''}(-1)^{\codim(E)}\right). 
\end{eqnarray*}
 We fix a stratum $E''$ in $\cX_k$ that contains $E'$. Since $\cX$ is strictly toroidal, there exists an inclusion preserving bijective correspondence between the strata $E$ in $\cX_k$ such that $E'\subset E\subset E''$, and the strict faces of a strictly convex rational polyhedral cone $\sigma$; the dimension of the face is equal to the dimension of the corresponding stratum minus $\dim(E')$.  It follows that  
 $$\sum_{E'\subset E \subset E''}(-1)^{\codim(E)}$$ is equal to $(-1)^{\codim(E')}$ times the compactly supported Euler characteristic of $\sigma$, which is $1$ if $E'=E''$ (then $\sigma$ has dimension $0$) and $0$ otherwise. Therefore, 
\begin{eqnarray*}
& &\sum_{E''\supset E'}\left([\mathbb{G}_{m,k}^{\codim(E'')}]_{e-\dim(E')}\sum_{E'\subset E \subset E''}(-1)^{\codim(E)}\right)
\\&=&(-1)^{\codim(E')}[\mathbb{G}_{m,k}^{\codim(E')}]_{e-\dim(E')}
\end{eqnarray*} and this implies the desired equality.
\end{proof}

The dimensional refinement of the motivic volume is characterized by the following theorem.

\begin{theorem}\label{theo:MV}
There exists a unique ring morphism 
$$\Vol\colon \Grodim(\Vardim_K)\to \Grodim(\Vardim_k)$$
such that for every strictly toroidal proper $R$-scheme $\cX$ with smooth generic fiber $X=\cX_K$, and for every integer $e\geq \dim(X)$, we have 
\begin{eqnarray}
\Vol([X]_e)&=& \sum_{E\in \mathcal{S}(\cX)}(-1)^{\codim(E)}[E^o\times_k \mathbb{G}^{\codim(E)}_{m,k}]_e \label{eq:motvol-open}
\\ &=& \sum_{E\in \mathcal{S}(\cX)}(-1)^{\codim(E)}[E]_{e-\codim(E)}P(E). \label{eq:motvol-closed}
\end{eqnarray}
\end{theorem}
\begin{proof}
The main difference with the set-up in Appendix A of \cite{NiSh} is that we do not have a globally defined log structure on $\cX$ such that the structural morphism to $\Spec R$ with its standard log structure is smooth. In the language of toroidal embeddings \cite{KKMS}, the problem can be described in the following way. Let $\cU$ be an open subscheme of $\cX$ as in \eqref{eq:tor} and let $R_0$ be a finite extension of $k\llbr t\rrbr$ in $R$ such that the morphism \eqref{eq:tor} is obtained from a smooth morphism of $R_0$-schemes  
$$\cU_0\to   \Spec R_0[M]/(t^{q}-\chi^m)$$ by extension of scalars to $R$. The pullback of the toric boundary on $ \Spec R_0[M]$ is a divisor $D$ on $\cU_0$ such that $\cU_0\setminus D\to \cU_0$ is a {\em toroidal embedding without self-intersection} over the discrete valuation ring $R_0$ in the sense of \cite{KKMS}. But the divisors $D$ do not necessarily glue to a divisor on the whole of $\cX$.

 To resolve this issue, we first define a local variant of the motivic volume. Let $\cY$ be a separated flat $R$-scheme of finite type of pure relative dimension $d$, with smooth generic fiber $\cY_K$. By resolution of singularities and the semi-stable reduction theorem, we can find a positive integer $n$, a model $\cY_0$ for $\cY$ over $R_0=k\llbr t^{1/n}\rrbr$ and a proper morphism of $R_0$-schemes  $h\colon \cZ\to \cY_0$ such that $h$ is an isomorphism on the generic fibers, $\cZ$ is regular, and $\cZ_k$ is a reduced divisor with strict normal crossings. Then we define the motivic volume of $\cY$ as
 $$\Vol(\cY)=\sum_{E\in \mathcal{S}(\cZ)}(-1)^{\codim(E)}[E^o\times_k \mathbb{G}^{\codim(E)}_{m,k}]_{d}\quad \in \Grodim(\Vardim_k).$$
 It can be deduced from the weak factorization theorem in \cite{AT} and some elementary calculations that this definition only depends on $\cZ$, and not on the choices of $n$, $\cY_0$ and $\cZ$: the arguments in Appendix A of \cite{NiSh} immediately carry over to our setting (see in particular Propositions A.3.5 and A.3.8). 
%In \cite{NiSh} the formula for the motivic volume was stated in terms of the {\em open} strata $E^o$, but it is equivalent to our formula by a straightforward inclusion-exclusion argument.
 It is clear from the definition that $\Vol(\cY)$ is local on $\cY$: for every finite open cover $\{\cU_{\alpha}\,|\,\alpha\in A\}$  of $\cY$, we have $$\Vol(\cY)=\sum_{\emptyset \neq B\subset A}(-1)^{|B|-1}\Vol(\cap_{\beta\in B}\cU_{\beta})$$ in $\Grodim(\Vardim_k)$.
 
 The next step is to prove that, when  $\cX$ is a {\em strictly toroidal} $R$-scheme of pure relative dimension $d$ with smooth generic fiber, we still have 
 $$\Vol(\cX)=\sum_{E\in \mathcal{S}(\cX)}(-1)^{\codim(E)}[E^o\times_k \mathbb{G}_{m,k}^{\codim(E)}]_d$$ 
 in $\Grodim(\Vardim_k)$.
 Since both sides of the expression are local on $\cX$, we may assume that there exists an \'etale morphism $\cX\to \Spec R[M]/(t^{q}-\chi^m)$ as in \eqref{eq:tor}, which descends to an \'etale morphism $\cX_0\to \Spec R_0[M]/(t^{q}-\chi^m)$ over some finite extension $R_0$ of $k\llbr t\rrbr$ in $R$. Denote by $D$ the pullback to $\cX_0$ of the toric boundary of $\Spec R_0[M]$.    Then the open embedding $\cX_0\setminus D\to \cX_0$ is a toroidal embedding without self-intersections in the sense of \cite{KKMS}, and one can use a suitable subdivision of the associated cone complex  to construct, over some finite extension of $R_0$, a proper morphism $\cZ\to \cX_0$ that is an isomorphism on the generic fibers and such that $\cZ$ is regular and $\cZ_k$ is a reduced divisor with strict normal crossings. A toric calculation shows that  the expression
 $$\sum_{E\in \mathcal{S}(\cX)}(-1)^{\codim(E)}[E^o\times \mathbb{G}_{m,k}^{\codim(E)}]_d$$ 
 remains invariant under this modification. This calculation is carried out in Section A.2 of \cite{NiSh} in the language of logarithmic schemes.

Now let $X$ be a connected smooth and proper $K$-scheme of dimension $d$, and let $\cX$  be a  strictly toroidal proper $R$-model of $X$. Since any pair of proper $R$-models of $X$ can be dominated by a common toroidal proper  $R$-model, the above arguments imply that $\Vol(\cX)$ only depends on $X$. Thus for every integer $e\geq d$,  we may define 
$$\Vol([X]_e)=\Vol(\cX)\tau^{e-d}=\sum_{E\in \mathcal{S}(\cX)}(-1)^{\codim(E)}[E^o\times_k \mathbb{G}_{m,k}^{\codim(E)}]_e.$$ 
The equality of the expressions \eqref{eq:motvol-open} and \eqref{eq:motvol-closed} follows from  Lemma \ref{lemm:incexc}.

 The final step of the proof is to show that $\Vol([X]_e)$ is multiplicative in $[X]_e$ and satisfies the blow-up relations in Bittner's presentation \ref{thm:dimBitt}. Multiplicativity follows easily from the fact that the product of two strictly toroidal $R$-schemes is again strictly toroidal (see \cite[A.3.7]{NiSh}). Let $Y$ be a connected smooth strict closed subscheme of $X$, and let 
 $\mathrm{Bl}_Y X\to X$  be the blow-up of $X$ along $Y$, with exceptional divisor $E$.
  We can find a proper strictly toroidal $R$-model $\cX$ of $X$ such that the schematic closure $\cY$ of $Y$ in $\cX$ has transversal intersections with the special fiber $\cX_k$. Then $\cY$ is a strictly toroidal proper $R$-model of $Y$. Moreover,  the  blow-up $\mathrm{Bl}_{\cY} \cX$ of $\cX$ along $\cY$ is a strictly toroidal proper  $R$-model of $\mathrm{Bl}_YX$, and the schematic closure of $E$ in $\mathrm{Bl}_{\cY} \cX$ is a strictly  toroidal proper $R$-model of $E$. Now a direct calculation shows that 
$$\Vol([X]_e)-\Vol([Y]_e)=\Vol([\mathrm{Bl}_YX]_e)-\Vol([E]_e)$$ in $\Grodim(\Vardim_k)$, for all integers $e\geq d$.
\end{proof}

\begin{corollary}\label{cor:spbir}
There exist unique ring morphisms 
$$\Vol_{\bir}\colon \Z[\Bir_K]\to \Z[\Bir_k],\quad \Vol_{\sbir}\colon \Z[\SB_K]\to \Z[\SB_k]$$
such that for every strictly toroidal proper $R$-scheme $\cX$ with smooth generic fiber $X=\cX_K$, we have 
\begin{eqnarray}
\Vol_{\bir}(\{X\}_{\bir})&=&\sum_{E\in \mathcal{S}(\cX)}(-1)^{\codim(E)}\{E\times_k \mathbb{P}_k^{\codim(E)}\}_{\bir},
\label{eq:birvol}
\\ 
\Vol_{\sbir}(\{X\}_{\sbir})&=&\sum_{E\in \mathcal{S}(\cX)}(-1)^{\codim(E)}\{E\}_{\sbir}.
\label{eq:sbvol}
\end{eqnarray}
 In particular, if $\cX$ is smooth and proper over $R$, then $\Vol_{\bir}(\{X\}_{\bir})=\{\cX_k\}_{\bir}$ and $\Vol_{\sbir}(\{X\}_{\sbir})=\{\cX_k\}_{\sbir}$.
\end{corollary}
\begin{proof}
Since $\Spec R$ is a strictly toroidal proper $R$-model of $\Spec K$,  the motivic volume $\Vol$ maps $[\Spec K]_d$ to $[\Spec k]_d$ for every $d\geq 0$. It follows from Proposition \ref{prop:grobir} that $\Vol$ factors through a ring morphism $\Vol_{\bir}$; it satisfies the formula in the statement because $\bir(P(E))=\{\mathbb{P}^{\codim(E)}_k\}_{\bir}$ for every stratum $E$ in $\cX_k$.
 Since $\mathbb{P}^n_R$ is a strictly toroidal proper $R$-model of $\mathbb{P}^n_K$ for every $n\geq 0$, the morphism $\Vol_{\bir}$ maps $\{\mathbb{P}^n_K\}_{\bir}$ to $\{\mathbb{P}^n_k\}_{\bir}$. Because $\Vol_{\bir}$ is also multiplicative, it factors through a morphism $\Vol_{\sbir}$ as in the statement of the corollary.  
\end{proof}

\section{Applications to rationality problems}\label{ss:apprat}
\subsection{Specialization of birational types}
A first application of Corollary \ref{cor:spbir} is that it settles the long-standing question of specialization of (stable) birational equivalence. The case of stable birational equivalence was first proved in \cite{NiSh}; the stronger result for birational equivalence follows from the results in \cite{KT}. This application only uses the special case of formulas \eqref{eq:birvol} and \eqref{eq:sbvol} where $\cX$ is smooth over $R$.

\begin{theorem}\label{theo:spec}
 Let $S$ be a Noetherian scheme of characteristic zero, and let $\cX\to S$ and $\cY\to S$ be smooth and proper $S$-schemes. For every point $s$ of $S$, we fix a geometric point $\overline{s}$ supported at $s$.
  We define subsets $S_{\bir}(\cX,\cY)$ and $S_{\sbir}(\cX,\cY)$ of $S$ in the following way: 
 \begin{eqnarray*}
 S_{\bir}(\cX,\cY)&=&\{s\in S\,|\,\cX\times_S\overline{s}\sim_{\bir} \cY\times_S \overline{s}\},
 \\ S_{\sbir}(\cX,\cY)&=& \{s\in S\,|\,\cX\times_S\overline{s}\sim_{\sbir} \cY\times_S \overline{s}\},
 \end{eqnarray*}
 where $\sim_{\bir}$ and $\sim_{\sbir}$ denote birational equivalence and stable birational equivalence, respectively. 
  Then $S_{\bir}(\cX,\cY)$ and $S_{\sbir}(\cX,\cY)$ are countable unions of closed subsets of $S$.
\end{theorem}
\begin{proof}
It follows from a standard Hilbert scheme argument that $S_{\bir}(\cX,\cY)$ and $S_{\sbir}(\cX,\cY)$ are countable unions of {\em locally} closed subsets of $S$; see for instance Proposition 2.3 in \cite{dFF}, which is stated in a more restrictive setting but whose proof also confirms our more general statement. Therefore, it is sufficient to prove that $S_{\bir}(\cX,\cY)$ and $S_{\sbir}(\cX,\cY)$ are closed under specialization. Now one easily reduces to the case where $S=\Spec k\llbr t\rrbr$ and $k$ is algebraically closed; see the proof of Theorem 4.1.4 in \cite{NiSh}. Then $\cX\times_S \Spec R$ and $\cY\times_S \Spec R$ are smooth and proper $R$-schemes, and Corollary \ref{cor:spbir} implies that $\cX_k$ is  birational (resp.~stably birational) to $\cY_k$ if $\cX_K$ is birational (resp.~stably birational) to $\cY_K$.
\end{proof}

\begin{corollary}\label{cor:spec}
 Let $S$ be a Noetherian scheme of characteristic zero, and let $\cX\to S$ be a smooth and proper $S$-scheme. For every point $s$ of $S$, we fix a geometric point $\overline{s}$ supported at $s$.
  We define subsets $S_{\mathrm{rat}}(\cX)$ and $S_{\mathrm{srat}}(\cX)$ of $S$ in the following way: 
 \begin{eqnarray*}
 S_{\mathrm{rat}}(\cX)&=&\{s\in S\,|\,\cX\times_S\overline{s}\mbox{ is rational}\},
 \\ S_{\mathrm{srat}}(\cX)&=& \{s\in S\,|\,\cX\times_S\overline{s}\mbox{ is stably rational}\}.
 \end{eqnarray*}
  Then $S_{\mathrm{rat}}(\cX)$ and $S_{\mathrm{srat}}(\cX)$ are countable unions of closed subsets of $S$.
\end{corollary}
\begin{proof}
This follows immediately from Theorem \ref{theo:spec}, because 
$$S_{\mathrm{rat}}(\cX)=\bigcup_{n\geq 0}S_{\bir}(\cX,\mathbb{P}^n_S)$$
and $S_{\mathrm{srat}}(\cX)=S_{\sbir}(\cX,S)$. If $S$ is connected, we simply have $S_{\mathrm{rat}}(\cX)=S_{\bir}(\cX,\mathbb{P}^n_S)$ with $n$ the dimension of the fibers of $\cX\to S$.
\end{proof}

Countable unions cannot be avoided in the statements of Theorem \ref{theo:spec} and Corollary \ref{cor:spec}: in \cite{HPT}, Hassett, Pirutka and Tschinkel have constructed a smooth and proper family $\cX\to S$ over a complex variety $S$ such that $S_{\mathrm{rat}}$ is dense in $S$ but $S_{\mathrm{srat}}\neq S$.

\subsection{Obstruction to stable rationality}
 By contraposition, we can also use Corollary \ref{cor:spbir} as an obstruction to rationality or stable rationality of $\cX_K$.
 
\begin{theorem}\label{theo:obstruct}
Let $\cX$ be a strictly toroidal proper $R$-scheme. If 
$$\sum_{E\in \mathcal{S}(\cX)}(-1)^{\codim(E)}\{E\times \mathbb{P}^{\codim(E)}_k\}_{\bir}\neq \{\Spec k\}_{\bir}$$
in $\Z[\Bir_k]$, then $\cX_K$ is not rational. 
Similarly, if 
$$\sum_{E\in \mathcal{S}(\cX)}(-1)^{\codim(E)}\{E\}_{\sbir}\neq \{\Spec k\}_{\sbir}$$
in $\Z[\SB_k]$, then $\cX_K$ is not stably rational. Here the sums are taken over the strata $E$ in $\cX_k$.
\end{theorem}
\begin{proof}
This follows immediately from Corollary \ref{cor:spbir}.
\end{proof}
 
 These obstructions are not always easy to use in practice, because one needs to control the cancellations in the alternating sums, and thus understand the (stable) birational equivalences between the individual strata. Let us look at an interesting special case where cancellations do not occur.

\begin{corollary}\label{coro:obstruct}
Let $\cX$ be a strictly toroidal proper $R$-scheme. Suppose that every connected component of every stratum $E$ of even (resp.~odd) codimension in $\cX_k$ is stably rational, and that at least one connected component of some stratum of odd (resp.~even) codimension in $\cX_k$ is not stably rational. Then $\cX_K$ is not stably rational.
\end{corollary}  
\begin{proof}
The assumption implies that all the terms appearing with a positive (resp.~negative) sign in the sum $$\sum_{E\in \mathcal{S}(\cX)}(-1)^{\codim(E)}\{E\}_{\bir}$$ are integer multiples  $\{\Spec k\}_{\sbir}$, while at least one term with opposite sign is not a multiple of $\{\Spec k\}_{\sbir}$.
 Thus the whole sum is different from $\{\Spec k\}_{\sbir}$.
\end{proof}

In order to apply Theorem \ref{theo:obstruct} and Corollary \ref{coro:obstruct} to find new classes of non-stably rational varieties, one always needs non-trivial input, namely, a strictly toroidal degeneration such that at least one stratum in the special fiber is not stably rational and such that one can control the potential cancellations in the alternating sum of stable birational types. A convenient method to produce interesting strictly toroidal degenerations is provided by tropical geometry; this method is used in \cite{NO} to obtain various new stable irrationality results. The upshot of this technique is that one can deduce stable irrationality of a very general member of a family of varieties from the stable irrationality of special varieties in lower dimensions and/or degrees.

\subsection{Examples}\label{ss:exam}
We will discuss a few applications of the obstruction to stable rationality in Theorem \ref{theo:obstruct}. 
Examples \ref{exam:double4solid}, \ref{exam:4four}, \ref{exam:five} and \ref{exam:BB} have already been obtained by different methods in the literature; here our aim is merely to illustrate the general technique. To the best of our knowledge, Examples \ref{exam:new} and \ref{exam:new2}  are new.  More elaborate applications can be found in \cite{NO}, where, among other results, we prove the stable irrationality of very general quartic fivefolds  and various new classes of complete intersections (including very general $(2,3)$ complete intersections in $\mathbb{P}^6$).

 Throughout this section, we denote by $k$ an algebraically closed field of characteristic zero. 

\begin{example}\label{exam:double4solid}
Our first example is taken from Theorem 4.3.1 in \cite{NiSh}. We will deduce from Theorem \ref{theo:obstruct} that a very general quartic double solid over $k$ is not stably rational; this is a special case of a result by Voisin \cite{voisin}. By Corollary \ref{cor:spec}, it suffices to construct one non-stably rational smooth quartic double solid over some algebraically closed field of characteristic zero; our base field will be the field $K$ of Puiseux series over $k=\C$.
 
 As input we use Artin and Mumford's famous example of a stably irrational  quartic double solid $Y_0$ over $\C$ with only isolated ordinary double points as singularities \cite{AM}. Let $F_0\in \C[z_0,\ldots,z_3]$ be a homogeneous degree $4$ polynomial that defines the ramification divisor $D_0$ of the double cover $Y_0\to \mathbb{P}^3_{\C}$. Let $F$ be a general homogeneous degree $4$ polynomial in $\C[z_0,\ldots,z_3]$. Let $\mathscr{D}$ be the divisor in $\mathbb{P}^3_{\C[t]}$ defined by $F_0-tF=0$ and let $\cY\to \mathbb{P}^3_{\C[t]}$ be the double cover ramified along $\mathscr{D}$. Then $\cY$ is a regular proper  $\C[t]$-scheme with special fiber $Y_0$; its generic fiber is a smooth quartic double solid. Let $\cY'\to \cY$ be the blow-up at the singular points of $Y+0$ and let $\cX$ be the normalization of $\cY'\times_{\C[t]}R$.
  By Example \ref{exam:snctor}, the $R$-scheme $\cX$ is strictly semi-stable. Its special fiber has a unique stably irrational stratum, namely, the strict transform of $Y_0$. Thus, it follows from Theorem \ref{theo:obstruct} that the smooth quartic double solid $\cX_K$ is stably irrational.
\end{example}

\begin{example}\label{exam:4four}
The next application concerns stable non-rationality of very general quartic hypersurfaces of dimensions $4$ and $5$. In the fourfold case, this was first proved by Totaro as a special case of the general bound he established in \cite{totaro}; this bound was further improved (and extended to positive characteristic) by Schreieder in \cite{schreieder}. The fivefold case was first proved in \cite{NO}, as as an application of the tropical techniques we develop in that paper. Here we will treat both cases in a uniform way, without invoking tropical methods.

Let $n\in \{4,5\}$ and let $F\in k[z_0,\ldots,z_{n+1}]$ be a homogeneous polynomial of degree $4$, which we choose to be very general subject to the condition that $F$ is invariant under the transposition of the variables $z_n$ and $z_{n+1}$. 
Set 
$$\cX=\mathrm{Proj}\,R[z_0,\ldots, z_{n+1},y]/(yt-z_nz_{n+1},y^2-F)$$ where the variable $y$ has weight $2$.  The generic fiber $\cX_K$ is a smooth quartic hypersurface in $\PP^{n+1}_K$ (we can make the substitution $y=z_nz_{n+1}/t$ because $t$ is invertible in $K$).

 The $R$-scheme $\cX$ is strictly toroidal: away from the locus $Z$ defined by $y=t=z_n=z_{n+1}=0$, the scheme 
 $$\cX'=\mathrm{Proj}\,k\llbr t\rrbr[z_0,\ldots,z_{n+1},y]/(yt-z_nz_{n+1},y^2-F)$$ is regular and its special fiber is a reduced divisor with strict normal crossings, so that $\cX=\cX'\times_{k[t]}R$ is strictly semi-stable away from $Z$. On the other hand, for every $i$  in $\{0,\ldots,n-1\}$, the projection morphism $$D_+(z_i)\to \Spec R\left[\frac{z_n}{z_i},\frac{z_{n+1}}{z_i},\frac{y}{z_i^2}\right]/\left(\frac{y}{z_i^2}t-\frac{z_n}{z_i}\frac{z_{n+1}}{z_i}\right)$$ is smooth along $Z\cap D_+(z_i)$, so that $\cX$ is also strictly toroidal along $Z$.

The special fiber $\cX_k$ has three strata. The two irreducible components $E_1$ and $E_2$, given by $z_n=0$ and $z_{n+1}=0$, are isomorphic because of the symmetry of $F$. When $n=4$, their intersection is a very general quartic double solid, and thus not stably rational by Example \ref{exam:double4solid}. When $n=5$, we conclude similarly using the result that very general quartic double fourfolds over $k$ are also not stably rational, by \cite{HPTdouble}.

 Thus, in each case, we have 
 $$\{E_1\}_{\sbir}+\{E_2\}_{\sbir}-\{E_1\cap E_2\}_{\sbir}=2\{E_1\}_{\sbir}-\{E_1\cap E_2\}_{\sbir}\neq \{\Spec k\}_{\sbir}$$ in $\Z[\SB_k]$, and it follows from Theorem \ref{theo:obstruct} that $\cX_K$ is not stably rational. Now Corollary \ref{cor:spec} implies that a very general quartic fourfold  or fivefold is not  stably rational.
%  Very general quartic double fourfolds over $k$ are also not stably rational, by \cite{HPTdouble}.
  %A similar construction then allows to deduce the stable irrationality of a very general quartic fivefold; this is explained in Theorem 4.1.1 of \cite{NO} as an application of the tropical techniques we develop in that paper.
\end{example}

\begin{example}\label{exam:five}
Next, we prove that very general sextic fivefolds and sixfolds are not stably rational. These cases  also fall in the range of results in \cite{totaro} and \cite{schreieder}.  
 Our starting points are the stable irrationality of very general complete intersections of three quadrics in $\mathbb{P}^6_{k}$ \cite[\S4.4]{HT} and in $\mathbb{P}^7_k$ \cite{HPT-3quad}.

Let $(Q_1,Q_2,Q_3)$ be a very general triple of quadratic forms in $k[z_0,\ldots,z_n]$, where $n$ is either $6$ or $7$. Let $F$ be a general sextic form in $k[z_0,\ldots,z_n]$. 
Then the $R$-scheme 
$$\cX=\mathrm{Proj}\,R[z_0,\ldots,z_n]/(tF-Q_1Q_2Q_3)$$ is strictly toroidal, by Example \ref{exam:tor}. Since smooth quadrics and smooth intersections of two quadrics in $\mathbb{P}^n_k$ are rational, the only non-stably rational stratum in $\cX_k$ is the triple intersection defined by 
$Q_1=Q_2=Q_3=0$.  Theorem \ref{theo:obstruct} now implies that $\cX_K$ is not stably rational, so that a very general sextic fivefold and sixfold are not stably rational by Corollary \ref{cor:spec}. 
 \end{example}

\begin{example}\label{exam:BB}
In this example, we will prove that for every $d\geq 2$, a very general hypersurface in $\mathbb{P}_k^2\times_k \mathbb{P}^2_k$ of bidegree $(2,d)$ is not stably rational. This was the main result in \cite{BB}.
 
 As input we use the property that very general hypersurfaces of bidegree $(2,2)$ in $\mathbb{P}_k^2\times_k \mathbb{P}^2_k$ are not stably rational by \cite[\S8.2]{HT}. This settles the $d=2$ case. We prove the general case by induction on $d$. Assume that $d>2$ and that the result holds for hypersurfaces of bidegree $(2,d-1)$. Let $F,G\in k[z_0,z_1,z_2,w_0,w_1,w_2]$ be very general bihomogeneous polynomials of bidegree $(2,d)$ and $(2,d-1)$, respectively. Consider the closed subscheme $\cX$ of $\mathbb{P}^2_R\times_R \mathbb{P}^2_R$ defined by $tF-w_2G=0$.  Then $\cX$ is strictly toroidal, by the same argument as in Example \ref{exam:tor}. The special fiber $\cX_k$ has three strata: the linear space in $\mathbb{P}^2_k\times_k \mathbb{P}^2_k$ defined by $w_2=0$; the very general bidegree $(2,d-1)$ hypersurface defined by $G=0$, which is not stably rational by the induction hypothesis; and their intersection given by $w_2=G=0$, which is a smooth bidegree $(2,d-1)$ hypersurface in $\mathbb{P}^2_k\times_k \mathbb{P}^1_k$. The latter is a conic bundle over $\PP^1_k$, which is rational by Tsen's theorem. Now Theorem \ref{theo:obstruct} implies that $\cX_K$ is not stably rational, so that a very general bidegree $(2,d)$ hypersurface in $\mathbb{P}^2_k\times_k \mathbb{P}^2_k$ is not stably rational by Corollary \ref{cor:spec}.
 \end{example}

As a further illustration, we will discuss two applications that have not yet appeared in the literature.
\begin{example}\label{exam:new}
The first new result is that a very general intersection of a bidegree $(1,2)$ hypersurface and a bidegree $(2,2)$ hypersurface in $\mathbb{P}_k^2\times_k \mathbb{P}^4_k$ is not stably rational. Such a variety is fibered in quartic del Pezzo surfaces {via} the projection to $\mathbb{P}^2_k$; not much appears to be known about stable rationality of del Pezzo fibrations over $\mathbb{P}^n_k$ for $n\geq 2$. 

 The argument is similar to that in Example \ref{exam:BB}. 
 Let $$F,G,H \in k[z_0,z_1,z_2,w_0,\ldots,w_4]$$ be very general bihomogeneous polynomials of bidegree $(1,2)$, $(1,2)$ and $(2,2)$, respectively. Then the closed subscheme of $\mathbb{P}^2_R\times_R \mathbb{P}^4_R$ defined by   $F=tH-z_2G=0$ is strictly toroidal, by a similar calculation as in Example \ref{exam:4four}.  The generic fiber $\cX_K$ is a smooth complete intersection of  a bidegree $(1,2)$ hypersurface and a bidegree $(2,2)$ hypersurface in $\mathbb{P}_K^2\times_K \mathbb{P}^4_K$. The special fiber $\cX_k$ contains three strata: 
 the two irreducible components $E_1$ and $E_2$, given by $\{F=G=0\}$ and $\{F=z_2=0\}$, respectively, and their intersection $E_1\cap E_2$. The component $E_1$ is a very general complete intersection of two bidegree $(1,2)$ hypersurfaces in $\mathbb{P}^2_k\times_k \mathbb{P}^4_k$; this is birational to $\mathbb{P}^4_k$ {via} the projection to the second factor. 
  The component $E_2$ is a smooth bidegree $(1,2)$ hypersurface in $\mathbb{P}^1_k\times_k \mathbb{P}^4_k$, and, therefore, a quadric bundle over $\PP^1_k$ (which is rational by Tsen's theorem).  
  
  The intersection $E_1\cap E_2$ is a very general intersection of two bidegree $(1,2)$ hypersurfaces in $\mathbb{P}^1_k\times_k \mathbb{P}^4_k$. Such an intersection is not stably rational by Theorem 2 in \cite{HT} (it is a very general quartic del Pezzo fibration over $\mathbb{P}^1_k$ with height invariant $h=20$). Now Corollary \ref{coro:obstruct} implies that $\cX_K$ is not stably rational.
  \end{example}

\begin{example}\label{exam:new2}
Using a similar construction, we can prove that a very general intersection of a bidegree $(1,1)$ hypersurface and a bidegree $(2,2)$ hypersurface in $\mathbb{P}_k^3\times_k \mathbb{P}^3_k$ is not stably rational. This fourfold is a conic bundle over $\mathbb{P}^3_k$ via the second projection.

Let $$F,G \in k[z_0,\ldots,z_3,w_0,\ldots,w_3]$$ be very general bihomogeneous polynomials of bidegree $(2,2)$ and $(1,1)$ respectively. As before, the closed subscheme of $\mathbb{P}^3_R\times_R \mathbb{P}^3_R$ defined by $F=tG-z_3w_3=0$ is strictly toroidal. The generic fiber $\cX_K$ is a smooth complete intersection of  a bidegree $(1,1)$ hypersurface and a bidegree $(2,2)$ hypersurface in $\mathbb{P}_K^3\times_K \mathbb{P}^3_K$. The special fiber $\cX_k$ has two irreducible components $E_1$ and $E_2$, given by $\{F=z_3=0\}$ and $\{F=w_3=0\}$. 
 The strata $E_1$ and $E_2$ are isomorphic to very general bidegree $(2,2)$ hypersurfaces in $\PP^2_k\times_k \PP^3_k$ and  $\PP^3_k\times_k \PP^2_k$, respectively; therefore, they are both non-stably rational by \cite{HPT}. Now the result follows from Theorem \ref{theo:obstruct}.  \end{example}

\section{The monodromy action}\label{sec:monodromy}
If $X$ is a scheme of finite type over $k\llpar t\rrpar$, rather than $K$, then the motivic volume of $X\times_{k\llpar t\rrpar}K$ defined in \cite{NiSh} carries additional structure: an action of the profinite group scheme $\hat{\mu}$ of roots of unity over $k$. This structure captures the monodromy action on the cohomology of $X$ and plays an important role in the theory of motivic Igusa zeta functions \cite{DL}. In this final section we will briefly explain how this structure can also be defined for our dimensional refinement of the motivic volume.

\subsection{The equivariant Grothendieck ring}
The equivariant version of the classical Grothendieck ring of varieties was defined by Denef and Loeser; we will follow the construction in Section 2.3 of \cite{NiSh}, which is the most appropriate for our purposes.

 Let $F$ be a field of characteristic zero, and let $G$ be a profinite group scheme over $F$. We say that a quotient group scheme $H$ of $G$ is {\em admissible} if the kernel of
$G(F^a)\to H(F^a)$ is an open subgroup of the profinite group $G(F^a)$, where $F^a$ denotes an algebraic closure of $F$. In particular, $H$ is a finite group scheme over $F$.

 Let $d$ be a non-negative integer. The Grothendieck group $\Gro^{G}(\Var^{\leqslant d}_F)$ of $F$-varieties of dimension at most $d$ with $G$-action is the abelian group with the following presentation:
\begin{itemize}
\item {\em Generators}: isomorphism classes $[X]_d$ of $F$-schemes $X$ of finite type and of dimension at most $d$ endowed with a good $G$-action. Here ``good'' means that the action factors through an admissible quotient of $G$ and that we can cover $X$ by $G$-stable affine open subschemes (the latter condition is always satisfied when $X$ is quasi-projective). Isomorphism classes are taken with respect to $G$-equivariant isomorphisms.
\item {\em Relations}: we consider two types of relations.
\begin{enumerate}
\item {\em Scissor relations}: if $X$ is a $F$-scheme of finite type of dimension at most $d$ with a good $G$-action and $Y$ is a $G$-stable closed subscheme of $X$, then
$$[X]_d=[Y]_d+[X\setminus Y]_d.$$
\item {\em Trivialization of linear actions}: let $X$ be a $F$-scheme of finite type with a good $G$-action, and let $V$ be a $F$-vector scheme of dimension $m$ with a good linear action of $G$. Assume that $\dim(X)+m\leq d$. 
 Then $$[X\times_F V]_d=[X\times_F \A^m_F]_d$$ where the $G$-action on $X\times_F V$ is the diagonal action and the action on $\A^m_F$ is trivial.
\end{enumerate}
\end{itemize}

We set 
$$\Grodim^G(\Var_F)=\bigoplus_{d\geq 0}\Gro^G(\Var_F^{\leqslant d}).$$
This graded abelian group has a unique graded ring structure such that 
$$[X]_d\cdot [X']_e=[X\times_F X']_{d+e}$$ 
for all $F$-schemes $X$, $X'$ of finite type of dimensions at most $d$ and $e$, respectively, and with good $G$-action. The $G$-action
on $X\times_F X'$ is the diagonal action. We write $\LL$ for the class $[\A^1_F]_1$ (where $\A^1_F$ carries the trivial $G$-action) in the ring $\Grodim^{G}(\Vardim_F)$.

 If $F'$ is a field extension of $F$, then we have an obvious base change morphism
 $$\Grodim^{G}(\Vardim_F)\to \Grodim^{G}(\Vardim_{F'}):[X]\mapsto [X\times_F F'].$$
If $G'\to G$ is a continuous morphism of profinite group schemes, then we can also consider the restriction morphism
$$\mathrm{Res}^G_{G'}:\Grodim^{G}(\Vardim_F)\to \Grodim^{G'}(\Vardim_F).$$
Both of these morphisms are ring homomorphisms.

\subsection{The monodromy action on the motivic volume}
Let $X$ be a smooth and proper $k\llpar t\rrpar$-scheme, and let $\cX$ be a proper flat $k\llbr t\rrbr$-scheme endowed with an isomorphism of $k\llpar t\rrpar$-schemes $\cX\times_{k\llbr t\rrbr}k\llpar t\rrpar\to X$. Assume that $\cX$ is regular and that the special fiber $\cX_k$ is a strict normal crossings divisor (not necessarily reduced). Then we call $\cX$ an {\em snc-model} of $X$.  We write 
$$\cX_k=\sum_{i\in I}N_iE_i$$ where $\{E_i\,|\,i\in I\}$ is the set of irreducible components of $\cX_k$, and the coefficients $N_i$ are their multiplicities.

 For every non-empty subset $J$ of $I$, we set 
 $$E_J=\bigcap_{j\in J}E_j,\quad E_J^o=E_J\setminus \left(\bigcup_{i\in I\setminus J}E_i\right).$$ 
 The scheme $E_J$ is a smooth and proper $k$-scheme of pure codimension $|J|-1$ in $\cX_k$,  and $E_J^o$ is a dense open subscheme of $E_J$. The subschemes $E_J^o$ form a partition of $\cX_k$.
  Let $n$ be the least common multiple of the multiplicities $N_i$, and let $\cY$ be the normalization of $\cX\times_{k\llbr t\rrbr}k\llbr t^{1/n}\rrbr$. An easy local calculation shows that $\cY$ is strictly toroidal. The $k$-group scheme $\mu_n$ of $n$-th roots of unity acts on $\Spec k\llbr t^{1/n}\rrbr$ {\em via} the morphism
$$k\llbr t^{1/n}\rrbr\to k[\zeta]/(\zeta^n-1)\otimes_k k\llbr t^{1/n}\rrbr$$ that maps a formal power series $\phi(t^{1/n})$ to $\phi(\zeta^{-1}t^{1/n})$. This induces an action of $\mu_n$ on $\cX\times_{k\llbr t \rrbr}k\llbr t^{1/n} \rrbr$, and also on its normalization $\cY$ because $\cY\times_k \mu_n$ is normal since $\mu_n$ is \'etale over $k$. 
 For every non-empty subset $J$ of $I$, the $\mu_n$-action on $\cY$ restricts to a $\mu_n$-action
on $\widetilde{E}_J=\cY\times_{\cX}E_J$ and $\widetilde{E}_J^o=\cY\times_{\cX}E^o_J$. By composition with the projection $\hat{\mu}\to \mu_n$ we obtain actions of $\hat{\mu}$ on $\widetilde{E}_J$ and $\widetilde{E}_J^o$, and it follows immediately from the construction that these actions are good. A more explicit description of the schemes $\widetilde{E}_J$ with their $\hat{\mu}$-actions can be found in Section 2.3 of \cite{tameram}; see also Section 4.1 in \cite{logzeta}.

\begin{theorem}\label{theo:MVmon}
There exists a unique ring morphism 
$$\Vol\colon \Grodim(\Vardim_{k\llpar t\rrpar})\to \Grodim^{\hat{\mu}}(\Vardim_{k})$$
such that, for every smooth and proper $K$-scheme $X$, every snc-model $\cX$ of $X$ with $\cX_k=\sum_{i\in I}N_iE_i$, and every integer $e\geq \dim(X)$, we have 
\begin{eqnarray}
\Vol([X]_e) &=& \sum_{\emptyset \neq J\subset I} (-1)^{|J|-1}[\widetilde{E}^o_J\times_k \mathbb{G}_{m,k}^{|J|-1}]_e \label{eq:volmu-open}
\\ &=&\sum_{\emptyset \neq J\subset I} (-1)^{|J|-1}[\widetilde{E}_J\times_k \mathbb{P}_k^{|J|-1}]_e \label{eq:volmu-closed}
\end{eqnarray}
where $\hat{\mu}$ acts trivially on the schemes $\mathbb{G}_{m,k}^{|J|-1}$ and $\mathbb{P}_k^{|J|-1}$. 
It fits into a commutative diagram 
$$\begin{CD} 
\Grodim(\Vardim_{k\llpar t \rrpar})@>\Vol>> \Grodim^{\hat{\mu}}(\Vardim_k)
\\ @VVV @VV \mathrm{Res}^{\hat{\mu}}_{\{1\}} V
\\ \Grodim(\Vardim_K)@>>\Vol>\Grodim(\Vardim_k) 
\end{CD}$$
where the left vertical arrow is the base change morphism and the right vertical arrow 
 forgets the $\hat{\mu}$-action. 
  
Moreover, for every integer $m>0$, we also have a commutative diagram
 $$\begin{CD} 
\Grodim(\Vardim_{k\llpar t \rrpar})@>\Vol>> \Grodim^{\hat{\mu}}(\Vardim_k)
\\ @VVV @VV\mathrm{Res}^{\hat{\mu}}_{\hat{\mu}(m)} V
\\ \Grodim(\Vardim_{k\llpar t^{1/m} \rrpar})@>>\Vol>\Grodim^{\hat{\mu}(m)}(\Vardim_k) 
\end{CD}$$
 where the left vertical arrow is the base change morphism and the right vertical arrow 
 restricts the $\hat{\mu}$-action to $\hat{\mu}(m)=\ker(\hat{\mu}\to \mu_m)$.
\end{theorem}
\begin{proof}
One can simply copy the proof of Theorem A.3.9 in \cite{NiSh}; all the arguments remain valid in the dimensional refinement of the Grothendieck ring. The equality of \eqref{eq:motvol-open} and \eqref{eq:motvol-closed} follows from the same calculation as in the proof of Lemma \ref{lemm:incexc}.
 The arguments in Section A.3 of \cite{NiSh} prove the existence and uniqueness of the morphism $\Vol$ in the statement, and its compatibility with base change to extensions $k\llpar t^{1/m}\rrpar$. The compatibility with the morphism $\Vol$ on $\Grodim(\Vardim_K)$ follows directly from the fact that $\cY\times_{k\llbr t^{1/n}\rrbr}R$ is a strictly toroidal proper $R$-model of $X\times_{k\llpar t\rrpar}K$.
\end{proof}

\providecommand{\arxiv}[1]{{\tt{arXiv:#1}}}

\end{document}